\documentclass{amsart}
\usepackage{amssymb,euscript,amsmath, mathrsfs, amscd, mathabx}
\usepackage[pdftex]{graphicx}
\usepackage[pdftex]{color}
\usepackage{bm}

\newcounter{ENUM}
\newcommand{\itm}{\item}

\newenvironment{Ilist}{\renewcommand{\theENUM}{\Roman{ENUM}}\renewcommand{\itm}{\addtocounter{ENUM}{1}\item[(\theENUM)]}\begin{itemize}\setcounter{ENUM}{0}}{\end{itemize}}
\newenvironment{alist}[1][0]{\renewcommand{\theENUM}{\alph{ENUM}}\renewcommand{\itm}{\addtocounter{ENUM}{1}\item[(\theENUM)]}\begin{itemize}\setcounter{ENUM}{#1}}{\end{itemize}}

\newcommand{\margh}[1]{}

\def\risom{\overset{\sim}{\rightarrow}}

\input xy
\xyoption{all}
\CompileMatrices

\def\ZZ{{\mathbb Z}}

\def\bn{{\bm{n}}}

\def\sE{{\mathscr E}}
\def\sF{{\mathscr F}}

\def\sL{{\mathscr L}}
\def\sM{{\mathscr M}}
\def\sO{{\mathscr O}}

\def\sV{{\mathscr V}}

\def\fg{{\mathfrak g}}

\def\e{\varepsilon}
\def\vp{\varphi}

\def\GL{\operatorname{GL}}

\def\Spec{\operatorname{Spec}}
\def\Pic{\operatorname{Pic}}

\def\id{\operatorname{id}}

\def\univ{\operatorname{univ}}

\def\rk{\operatorname{rk}}
\def\im{\operatorname{im}}
\def\ord{\operatorname{ord}}

\def\LG{\operatorname{LG}}

\newtheorem{thm}{Theorem}[section]
\newtheorem{prop}[thm]{Proposition}
\newtheorem{lem}[thm]{Lemma}
\newtheorem{cor}[thm]{Corollary}

\theoremstyle{definition}
\newtheorem{defn}[thm]{Definition}

\newtheorem{ex}[thm]{Example}
\newtheorem{sit}[thm]{Situation}
\theoremstyle{remark}
\newtheorem{notn}[thm]{Notation}
\newtheorem{rem}[thm]{Remark}
\newtheorem{warn}[thm]{Warning}
\numberwithin{equation}{section}

\begin{document}
\title{Limit linear series for curves not of compact type}
\author{Brian Osserman}
\begin{abstract} 
We introduce a notion of limit linear series for nodal curves which are
not of compact type. We give a construction of a moduli space of limit 
linear series, which works also in smoothing families, and we prove a 
corresponding specialization result. For a more restricted class
of curves which simultaneously generalizes two-component curves and
curves of compact type, we give an equivalent definition of limit linear
series, which is visibly a generalization of the Eisenbud-Harris 
definition. Finally, for the same class of curves, we prove a smoothing
theorem which constitutes an improvement over known results even in the
compact-type case.
\end{abstract}

\thanks{The author was partially supported by NSA grant H98230-11-1-0159
and Simons Foundation grant \#279151 during the preparation of this work.}
\maketitle

\section{Introduction}

The 1980's saw spectacular progress in the theory of linear 
series on curves and their applications, including the proofs of the
Brill-Noether (Griffiths-Harris \cite{g-h1}) and Gieseker-Petri 
(Gieseker \cite{gi1}) theorems, new results on the geometry of general 
linear series (Eisenbud-Harris \cite{e-h4}), and the proof that 
moduli spaces of curves of sufficiently high genus are of general
type (Harris-Mumford \cite{h-m2} and Eisenbud-Harris \cite{e-h6}).
What these results all had in common was that they made central use
of degeneration techniques, studying what happens to linear series as 
smooth curves degenerate to singular ones. Ultimately, Eisenbud and Harris 
developed a general theory of ``limit linear series'' for
curves of compact type, meaning those curves whose dual 
graphs are trees, or equivalently, whose Jacobians are compact.

For more than 25 years, the question of how to extend the Eisenbud-Harris
theory to curves not of compact type has remained open. Aside from the
intrinsic appeal of the question, there are various reasons one would
like to have such a theory:
\begin{itemize}
\item it would offer the most systematic
approach to computing the cohomology classes of higher-codimension
Brill-Noether classes on moduli spaces of curves;
\item it would allow greater
flexibility in choosing a degeneration to approach open questions such as
the maximal rank conjecture;
\item and it likewise offers a more general setting
for analyzing generic fibers of specific families of curves. For instance,
degenerations arising from considering modular curves in positive 
characteristic are often two-component nodal curves.
\end{itemize}
The question of limit linear series for curves not of compact type has 
been explored by Esteves in various papers, most notably with
Medeiros in \cite{e-m1}, but to date, no one has been able to develop a 
complete theory generalizing that of Eisenbud and Harris. Recently,
Amini and Baker \cite{a-b1} have proposed a notion of limit linear series
based on Brill-Noether theory for graphs, which they show generalizes
the definition of Eisenbud and Harris. However, while they prove a 
specialization theorem, it is not clear how to prove a smoothing theorem
for Amini-Baker limit linear series, or how to construct moduli spaces.

In the present paper, we propose a different notion of limit linear series
for curves not of compact type, developed independently and motivated in 
part by work of the author in higher rank \cite{os20}. After giving the 
definition, we construct moduli spaces both over individual curves and in 
smoothing families, and use them to prove a specialization result. We then
show that our definition is a generalization of the Eisenbud-Harris
definition. In fact, we do considerably more: 
for the class of curves of ``pseudocompact type,'' which is a simultaneous 
generalization of curves of
compact type and curves with two components (see Figure \ref{fig} below), 
we give an equivalent 
formulation which visibly generalizes the Eisenbud-Harris definition.
In essence, our more general definition is well-suited for abstract
theory and constructions, while the second definition is more 
tractable for computations.
Finally, for curves of pseudocompact type we prove a smoothing theorem, 
which is an improvement even for the compact-type case because it does
not only apply to refined limit linear series. 

To apply our smoothing
theorem, it is necessary to produce families of limit linear series 
having the expected dimension, and accordingly in \cite{os23} we 
carry out dimension counts. Using our generalized Eisenbud-Harris
definition, we show that for curves of pseudocompact type
the expected dimension of spaces of limit linear series is always correct,
in the sense that if certain gluing conditions impose the maximal 
codimension, then the dimension agrees with the Brill-Noether number $\rho$.
We also investigate several families of curves for which we can show the
gluing conditions do indeed impose the maximal codimension, giving in
particular new criteria for the generic fiber of a one-parameter family
of curves to be Brill-Noether general. One of the families we consider
in \cite{os23} is a broad generalization of the curves considered by Cools, 
Draisma, Payne and Robeva in the graph-theoretic context in \cite{c-d-p-r}, 
and we are able to use our theory to shed new light on their results,
and to suggest further directions of investigation for the Brill-Noether 
theory of graphs.
The relationship to the Amini-Baker theory will be investigated more
thoroughly in \cite{os24}, but in essence our approach keeps track of
more gluing data, while minimizing the role of graph theory. Although 
this may in principle make computations more
difficult, in practice this may not be the case, and we have found that
our approach has the desired dimension behavior in some cases (such as
binary curves) for which the Amini-Baker theory does not.

We now explain the basic ideas that go into our definition of limit linear
series. Suppose that $B=\Spec R$ with $R$ a discrete valuation ring, and 
$X \to B$ is a family of curves over $B$ with smooth generic fiber and 
reducible nodal special fiber $X_0$. Further suppose that the total space 
$X$ is regular. Then each component $Z_v$ of $X_0$ is a (Cartier) divisor 
on $X$, so any given extension of a line bundle $\sL_{\eta}$ on the generic 
fiber can be twisted by $\sO_X(Z_v)$ to obtain an infinite family of 
extensions.  Given $V_{\eta}$ an $(r+1)$-dimensional space of global 
sections of 
$\sL_{\eta}$, for any extension of $\sL_{\eta}$ there is a unique extension
of $V_{\eta}$. The idea introduced by Eisenbud and Harris was to use these
twists to concentrate multidegree on each component $Z_v$ of $X_0$, and 
then to restrict the resulting extension of $(\sL_{\eta},V_{\eta})$ to
$Z_v$, thereby obtaining a collection $(\sL^v,V^v)_v$ of $\fg^r_d$s on the 
components of $X_0$. The question then becomes to understand which such 
tuples of $\fg^r_d$s can arise as a limit in this way. Eisenbud and Harris 
found a compatibility condition in terms of vanishing sequences at the 
nodes, and used this to define their notion of limit linear series. 
The power of their definition was that it 
was fundamentally inductive, describing limit linear series almost
independently on each component, and thereby making computations very
tractable. However, the drawback of their definition was that it was
difficult to generalize, and also to use for more theoretical purposes, 
such as moduli space constructions. 

In \cite{os20}, it was shown that one can state an
equivalent formulation of the Eisenbud-Harris definition as follows:
if $w$ is a multidegree of total degree $d$ on $X_0$, and $\sL_w$ 
denotes the extension of $\sL_{\eta}$ having multidegree $w$, then
$\Gamma(X_0,\sL_w|_{X_0})$ contains the extension of $V_{\eta}$, and
must therefore have dimension at least $r+1$. Moreover, it is 
straightforward to see that this extension of $V_{\eta}$ may be 
obtained by gluing together sections from the various $V^v$. This 
leads to a definition of limit linear series as a generalized 
determinantal locus (Definition \ref{def:lls} below), which yields new 
moduli space constructions, and
which also lends itself to generalization to curves not of compact type.
The other basic ingredient of our definition is that we allow for
insertion of chain of rational curves at nodes, and keep track of 
information on these curves as combinatorially as possible, only 
considering spaces of global sections on the original components. 

We next discuss our equivalent definition, generalizing the Eisenbud-Harris
definition to a broader class of curves. We begin by recalling their 
definition. Given a tuple $(\sL^v,V^v)$ of $\fg^r_d$s on the components 
$Z_v$ of $X_0$, Eisenbud and Harris define the tuple to be a \textbf{limit 
linear series} if the following condition is satisfied: for every node of 
$X_0$, given as $Z_v \cap Z_{v'}$, write $a^v_0,\dots,a^v_r$ and 
$a^{v'}_0,\dots, a^{v'}_r$ for the vanishing sequences of $(\sL^v,V^v)$ 
(respectively, $(\sL^{v'},V^{v'})$) at the node in question; then we 
require
\begin{equation}\label{eq:eh}
a^v_j + a^{v'}_{r-j} \geq d
\end{equation}
for $j=0,\dots,r$.

Our generalized definition builds on this by replacing the vanishing
sequence with a ``multivanishing sequence'' which keeps track of 
vanishing at several points at a time, and by adding a gluing condition
on the spaces $V^v$, which is vacuously satisfied in the compact type
case. There are additional complications arising from keeping track of
potential chains of rational curves inserted at each node, but we illustrate 
the main ideas in the simplest case, where we have two components, and do
not insert any additional rational curves. 

Some preliminary definitions are as follows.

\begin{notn} Let $X$ be a smooth projective curve, $D$ an effective
divisor on $X$, and $(\sL,V)$ a $\fg^r_d$ on $X$. Then we denote by
$V(-D)$ the space $V \cap H^0(X,\sL(-D))$.
\end{notn}

\begin{defn}\label{def:multivanishing} Let $X$ be a smooth projective
curve, $r,d \geq 0$, and $D_0 \leq D_1 \leq \dots \leq D_{b+1}$ a 
sequence of effective divisors on $X$, with $D_0=0$ and $\deg D_{b+1}>d$. 
Given $(\sL,V)$ a $\fg^r_d$ on $X$, define the
\textbf{multivanishing sequence} of $(\sL,V)$ along $D_{\bullet}$ to
be the sequence
$$a_0 \leq \dots \leq a_r$$
where a value $a$ appears in the sequence $m$ times if for some $i$ we
have $\deg D_i=a$, $\deg D_{i+1}>a$, and 
$\dim \left(V(-D_i)/V(-D_{i+1})\right)=m$.

Also, given $s \in V$ nonzero, define the \textbf{order of vanishing}
$\ord_{D_{\bullet}} s$ along $D_{\bullet}$ to be $\deg D_i$, where
$i$ is maximal so that $s \in V(-D_i)$.
\end{defn}

Thus, multivanishing sequences generalize usual vanishing sequences and
ramification, incorporating also geometric notions such as secancy
conditions (requiring two or more points to map to a single a point),
bitangency, and so forth. Similar conditions for the case of rational 
curves were studied by Garc\'ia-Puente \textit{et al} in \cite{g-h-h-m-r-s-t}.
In \cite{os23} we observe that the standard results on Brill-Noether theory
with imposed ramification generalize to imposed multivanishing sequences.

Note that due to our choice of indexing of the multivanishing sequence, adding
repeated divisors into $D_{\bullet}$ does not affect the sequence.

Now, suppose that $X_0$ is obtained by gluing together smooth curves
$Z_1$ and $Z_2$ at nodes $P_1,\dots,P_m$. Given $d>0$, let $d_1,d_2$ be 
positive integers such that there exists $b \geq 0$ with $d=d_1+d_2-bm$,
and suppose also that $d-d_i<m$ for $i=1,2$ (in the Eisenbud-Harris case,
we will have $d_1=d_2=b=d$). For $i=1,2$ and $0 \leq j \leq b+1$, set 
$D^i_j=j (P_1+\dots+P_m)$.
Now, suppose we are given $(\sL^i,V^i)$ a $\fg^r_{d_i}$ on $Z_i$ for $i=1,2$, 
and suppose we are also given gluing information $\vp$
for $\sL^1$ and $\sL^2$ at the nodes. Then we define the tuple 
$((\sL^1,V^1),(\sL^2,V^2),\vp)$ to be a limit linear series if the
following two conditions are satisfied:
\begin{Ilist}
\itm for $i=1,2$, write $a^i_0,\dots,a^i_r$ for the multivanishing sequence 
of $(\sL^i,V^i)$ along $D^i_{\bullet}$; then we require
\begin{equation}\label{eq:eh-2} a^1_j+a^2_{r-j} \geq bm\end{equation}
for $j=0,\dots,r$;
\itm
for $i=1,2$, there exist bases $s^i_0,\dots,s^i_r$ of the $V^i$ such that
$$\ord_{D_{\bullet}}s^i_{\ell} = a^i_{\ell} \quad \text{ for } 
\ell=0,\dots,r,$$
and for all $\ell$ with \eqref{eq:eh-2} an equality, we have
$$\vp(s^1_{\ell})=s^2_{r-\ell}.$$
\end{Ilist}
In the above, we have been a bit vague in discussing the gluing; this is
made fully precise in \S \ref{sec:two-comp} below. Then, in
\S \ref{sec:pseudocompact-type}, we generalize to the case of curves of
pseudocompact type, meaning that if we take the dual graph,
and collapse all multiple edges, we obtain a tree;\footnote{More precisely,
we obtain a tree from the dual graph by, for each pair of adjacent vertices
$v,v'$, replacing all edges connecting $v$ to $v'$ with a single edge.} 
see Figure \ref{fig}.
\begin{figure}\label{fig}
\centering
\includegraphics{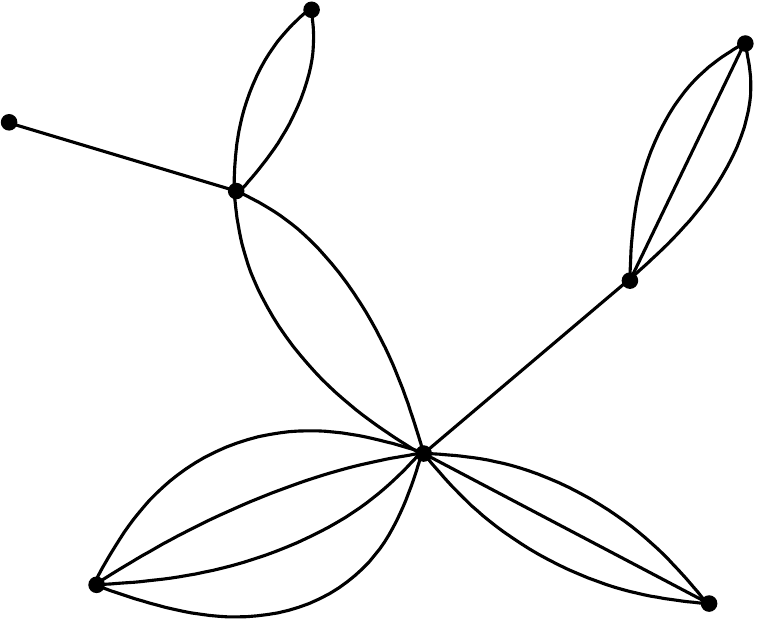}
\caption{A dual graph of a curve of pseudocompact type.}
\end{figure}
In addition to curves
of compact type, this includes interesting classes of curves such as 
curves with two components, and chains of curves of the sort considered 
by Cools, Draisma, Payne and Robeva in \cite{c-d-p-r}. One can think of
curves of pseudocompact type as being the most general class of curves
for which one can still analyze gluing conditions by looking at only
two components at a time. Note that there is a close parallel between
the above conditions (I) and (II) and the definition of limit linear
series for higher-rank vector bundles given by Teixidor i Bigas in
\cite{te1}. This parallel persists, albeit to a lesser extent, when
we allow insertions of chain of rational curves, and consider 
arbitrary curves of pseudocompact type. However, this is reflective
of the node-by-node aspect of the gluing conditions, and for arbitrary 
nodal curves the behavior is expected to be quite different.

Finally, in \S \ref{sec:smoothing} we prove the following smoothing 
theorem. We state it informally here, with a more precise statement as 
Theorem \ref{thm:smoothing} below.

\begin{thm}\label{thm:smoothing-intro}
If $X_0$ is a curve of pseudocompact type, and the space of limit linear
on $X_0$ has the expected dimension 
$$\rho:=g+(r+1)(d-r-g),$$ 
then every limit linear series
on $X_0$ can be smoothed to linear series on all nearby smooth curves. 
\end{thm}

As mentioned above, in comparison to the smoothing theorem of 
Eisenbud-Harris, our result is stronger because it is not confined to
the open subset of refined limit linear series.
The main tool in the proof of Theorem \ref{thm:smoothing-intro} is the
theory of linked determinantal loci, which we develop in Appendix 
\ref{sec:link-det}.

We conclude with a brief explanation of some of the decisions behind our
definitions. First, we originally intended to use torsion-free sheaves
to treat specializations, rather than allowing the insertion of rational 
curves at nodes. However, we discovered that from this point of view, 
important gluing conditions are omitted, and as a result, the spaces may 
no longer have the correct dimension.
Next, of course in a general theory of limit
linear series, in principle one does not need to treat inserted rational
chains differently from other components. However, there are two compelling 
reasons for doing so. The first is that it keeps the amount of data more
manageable; for instance, in the two-component case, we can study limit
linear series in general without having to remember more than two linear
series, one for each of the original components. The other reason is that
the pseudocompact type condition is not preserved under insertion of 
rational curves at nodes, so our second definition would not be complete
(for instance, with respect to specialization results) if we did not have
a system for keeping track of inserted rational curves. In addition, our
approach is very convenient for working with non-regular smoothing families.
The final comment
is that we have not, for the moment, pursued the possibility of creating
a single proper moduli space of limit linear series using the quasistable
curve compactification of the Picard variety. This is a natural and
worthwhile direction to pursue, but because Eisenbud and Harris were able
to carry out all their applications without a compact moduli space (using
instead a specialization result analogous to our Corollary 
\ref{cor:specialize}), it does not seem to be crucial to the basic theory.

\subsection*{Acknowledgements}

I would like to thank Eduardo Esteves for many helpful conversations,
particularly in relation to chain structures and admissible multidegrees.
I would also like to thank Frank Sottile for drawing my attention to
\cite{g-h-h-m-r-s-t}, and Ulrich Goertz for his assistance with the proof
of Proposition \ref{prop:matrix-prod-smooth}.

\subsection*{Conventions}

All curves we consider are assumed proper, (geometrically) reduced and
connected, and at worst nodal. All nodal curves are assumed to be split, 
meaning that both its nodes and irreducible components are all defined 
over the base field. Furthermore, to reduce clutter, we assume that all 
irreducible components are smooth; see Remark \ref{rem:self-nodes}.

A graph by default is allowed to have multiple edges, but not, in
accordance with the above, loops.

\section{Fundamental definitions}

We begin with some definitions of a combinatorial nature. In the below,
$\Gamma$ will be obtained by choosing a directed structure on the dual
graph of a projective nodal curve. We assume we have:

\begin{sit}\label{sit:gamma} Let $\Gamma$ be a directed graph without
loops. For each pair of an edge $e$ and adjacent vertex $v$ of $\Gamma$, 
let $\sigma(e,v)=1$ if $e$ has tail $v$, and $-1$ if $e$ has head $v$.
\end{sit}

The following definitions form the basis for our approach to keeping 
track of chains of rational curves inserted at the nodes of the original
curve.

\begin{defn} A \textbf{chain structure} on $\Gamma$ is a function 
$\bn:E(\Gamma) \to \ZZ_{> 0}$. A chain structure is \textbf{trivial}
if $\bn(e)=1$ for all $e \in E(\Gamma)$.
\end{defn}

The chain structure will determine the length of the chain of rational
curves inserted at a given node; for reasons of later convenience, the
trivial case (in which no rational curves are inserted) corresponds to
$\bn(e)=1$.

\begin{defn} Given $\bn$ a chain structure on $\Gamma$, an 
\textbf{admissible multidegree} $w$ of total 
degree $d$ on $(\Gamma,\bn)$ consists of a function 
$w_{\Gamma}: V(\Gamma) \to \ZZ$ together with a tuple
$(\mu(e))_{e \in E(\Gamma)}$, where each $\mu(e) \in \ZZ/\bn(e)\ZZ$,
such that
$$d = \#\{e \in E(\Gamma): \mu(e) \neq 0\}
+ \sum_{v \in V(\Gamma)} w_{\Gamma}(v).$$
\end{defn}

The idea behind admissible multidegrees is that in order to extend line
bundles, we need only consider multidegrees which have degree $0$ or $1$
on each rational curve inserted at the node, with degree $1$ occurring at
most once in each chain. Thus, $\mu(e)$ determines where on the chain
(if anywhere) positive degree occurs. 
See Definition \ref{def:multideg-lb} below for details.

\begin{defn} Given a chain structure $\bn$ on $\Gamma$,
let $w$ be an admissible multidegree. Given also
$v \in V(\Gamma)$, the \textbf{twist} of $w$ at $v$ is obtained as follows:
for each $e$ adjacent to $v$, increase $\mu(e)$ by $\sigma(e,v)$. Now,
decrease $w_{\Gamma}(v)$ by the number of $e$ for which $\mu(e)$ had
been equal to $0$, and for each $e$, if the new $\mu(e)$ is zero, increase
$w_{\Gamma}(v')$ by $1$, where $v'$ is the other vertex adjacent to $v$.
The \textbf{negative twist} of $w$ at $v$ is the admissible multidegree
$w'$ such that the twist of $w'$ at $v$ is equal to $w$.
\end{defn}

Twists will be the change in multidegrees accomplished by twisting by
certain natural line bundles; see Notation \ref{not:twist} below.

\begin{ex}\label{ex:twists-trivial} In the case of trivial chain structure,
a twist at $v$ simply reduces $w_{\Gamma}(v)$ by the valence of 
$v$ while increasing $w_{\Gamma}(v')$ by the number of edges connecting
$v'$ to $v$, for each $v' \neq v$. This is the same as the chip firing 
considered by Baker and Norine in \cite{b-n1}.
\end{ex}

\begin{rem}\label{rem:twists} Given $\Gamma$ and $\bn$, let 
$\widetilde{\Gamma}$ be the (directed) graph obtained from $\Gamma$ 
by subdividing each edge $e$ into $\bn(e)$ edges. Thus, we have
a natural inclusion $V(\Gamma) \subseteq V(\widetilde{\Gamma})$.
Then if $w$ is an admissible multidegree for $(\Gamma,\bn)$, we obtain 
a weight function $\widetilde{w}:V(\widetilde{\Gamma}) \to \ZZ$ on 
$\widetilde{\Gamma}$ (which we think of as being a multidegree for the
trivial chain structure) by setting $\widetilde{w}(v)=w_{\Gamma}(v)$ for
all $v \in V(\Gamma)$, and setting $\widetilde{w}(v)=0$ for all $v
\not \in V(\Gamma)$, unless $v$ lies over an edge $e$ of $\Gamma$, and
is the $\mu(e)$th new vertex lying over $e$. In the latter case, we set
$\widetilde{w}(v)=1$.

Thus, admissible multidegrees for $(\Gamma,\bn)$ are imbedded into the
set of multidegrees on $\widetilde{\Gamma}$, and this imbedding is
compatible with twists as follows: twisting $w$ at $v \in V(\Gamma)$ is
the same as twisting $\widetilde{w}$ by $v$, and then also by all
new vertices between $v$ and the $\sigma(e,v)\mu(e)$th new vertex lying over 
$e$, for each $e \in E(\Gamma)$ adjacent to $v$.
In the above, we take the representative of $\sigma(e,v)\mu(e)$ between $0$ 
and $\bn(e)-1$.
See also Notation \ref{not:twist} below for the geometric version
of this statement.
\end{rem}

\begin{ex}\label{ex:twists-2-comps}
In the two-component case, with components $v_1$ and $v_2$, and edges
oriented from $v_1$ to $v_2$, we describe twists in terms of 
multidegrees on $\widetilde{\Gamma}$ as in Remark \ref{rem:twists}.
The idea is that twisting by $v_1$ moves the positive-degree new vertices
away from $v_1$ and towards $v_2$.
Specifically, when twisting $w$ at $v_1$, for each $e \in E(\Gamma)$,
the degree-$1$ new vertex over $e$ shifts by one away from $v_1$.
If the vertex with degree $1$ is already adjacent to $v_2$, then
the degree on $v_2$ is increased, and no new vertices over $e$ will have
positive degree. If no new vertices over $e$ have degree $1$, then
the degree on $v_1$ is decreased, and the first new vertex over $e$ is
given degree $1$.
\end{ex}

Note that twists are invertible, since twisting at every vertex of 
$\Gamma$ returns to the initial multidegree. Thus, the negative twist
at $v$ can be expressed also as the composition of the twists at all
$v' \neq v$. We will primarily be interested in (positive) twists, but
the utility for us of negative twists is in the following definition.

\begin{defn} An admissible multidegree $w$ is \textbf{concentrated} at
a vertex $v \in V(\Gamma)$ if there is an ordering on $V(\Gamma)$ starting
with $v$, and such that for each subsequent vertex $v'$, we have that $w$
becomes negative in index $v'$ after taking the composition of the negative 
twists at all the previous vertices.
\end{defn}

A more canonical condition which
implies concentration (but is in general strictly stronger) is that 
for all $v' \neq v$, and all $v''$ adjacent to $v'$, the negative twist
of $w$ at $v''$ is negative in index $v'$. We have elected to use the above
definition as the most general for which one can make the argument of 
Proposition \ref{prop:concentrated-inject} below.

\begin{ex}\label{ex:concentrate} The concentration condition is the
generalization of the multidegrees considered by Eisenbud and Harris
in the compact type case, where they had degree $d$ on one component,
and degree $0$ on all the others. 

In our generalized setting, $w$ will be concentrated at $v$ if it is 
negative on all $v' \neq v$. If the chain structure is trivial, it is
enough to have degree at most $0$ at all $v' \neq v$, but in general
this is not the case, since with nontrivial chain structures, a negative 
twist at $v''$ adjacent to $v'$ need not reduce the degree on $v'$. 

However, at the opposite extreme, even with nontrivial chain structures
we can have a multidegree simultaneously concentrated at two adjacent
vertices. For instance, if $\Gamma$ has only two vertices, connected by
$n$ edges, then a multidegree which is strictly less than $n$ on each
vertex, and with $\mu(\bullet)$ identically zero, will be concentrated
on both vertices. This corresponds to usual linear series (of restricted
multidegrees) on the relevant two-component curves.
\end{ex}

See also Remark \ref{rem:v-reduced} below for further comments on the
role of the concentration condition.

\begin{prop}\label{prop:concentrate}
Given any admissible multidegree $w$, and any $v \in V(\Gamma)$, there
exists an admissible multidegree $w'$, concentrated at $v$, and obtained
from $w$ by repeated twisting at vertices $v'$ other than $v$.
\end{prop}

\begin{proof} First note that the composition of negative twists over a 
collection $S$ of vertices of $\Gamma$ is equivalent to the composition of
(positive) twists over the complement of $S$.
For each $n \geq 0$, let
$\Gamma_{v,n}$ denote the subset of $V(\Gamma)$ consisting of all vertices
$v'$ such that there is a path (undirected) in $\Gamma$ of length less than 
or equal to $n$ from $v$ to $v'$. Let $N$ be maximal such that 
$\Gamma_{v,N} \subsetneq V(\Gamma)$. Taking sufficiently many negative
twists of $w$ at all vertices of $\Gamma_{v,N}$, we can achieve negative
degrees at all vertices of $V(\Gamma) \smallsetminus \Gamma_{v,N}$.
Repeating this process for $\Gamma_{v,N-1}$ achieves negative degree
on $\Gamma_{v,N} \smallsetminus \Gamma_{v,N-1}$ without affecting the
degree on $V(\Gamma) \smallsetminus \Gamma_{v,N}$, and continuing in
this way down to $\Gamma_{v,0}$, we achieve negative degree at all
vertices other than $v$, which in particular implies concentration at $v$.
\end{proof}

The following directed graph keeps track of all the multidegrees we will 
want to consider starting from any one admissible multidegree.

\begin{notn} Let $G(w_0)$ be the directed graph with vertex set
$$V(G(w_0)) \subseteq \ZZ^{V(\Gamma)} \times \prod_{e \in E(\Gamma)} 
\ZZ/\bn(e)\ZZ$$
consisting of all admissible multidegrees obtained from $w_0$ by sequences 
of twists, and with an edge from $w$ to $w'$ if $w'$ is obtained from $w$ 
by twisting at some vertex $v$ of $\Gamma$.

Given $w \in V(G(w_0))$ and $v_1,\dots,v_m \in V(\Gamma)$ (not necessarily
distinct), let $P(w,v_1,\dots,v_m)$ denote the path in $V(G(w_0))$ obtained
by starting at $w$, and twisting successively at each $v_i$.
\end{notn}

By the invertibility of 
twists, $G(w_0)=G(w)$ if and only if $w \in G(w_0)$. While our directed 
structure on $\Gamma$ is just a convenience, the directedness of $G(w_0)$
is crucial. Although it is not
important for our present purposes, we also mention that $G(w_0)$ can be 
expressed as the collection of admissible multidegrees which are linearly
equivalent to $w_0$ on $\widetilde{\Gamma}_0$, using the theory of
linear equivalence on graphs as developed by Baker and Norine in 
\cite{b-n1}.

Also, note that $P(w,v_1,\dots,v_m)$ is independent of the ordering of the
$v_i$.

\begin{prop}\label{prop:paths-unique} If $P(w,v_1,\dots,v_m)$ is a minimal
path in $G(w_0)$ from $w$ to some $w'$, then $m$ and the $v_i$ are 
uniquely determined up to reordering. 

More generally, paths $P(w,v_1,\dots,v_m)$ and $P(w,v'_1,\dots,v'_{m'})$
have the same endpoint if and only if the multisets of the $v_i$ and the 
$v'_i$ differ by a multiple of $V(\Gamma)$.
\end{prop}

\begin{proof} We have already observed the ``if'' direction. For the
converse, in light of Remark \ref{rem:twists}
the desired statement for $\Gamma$ and $\bn$ follows from the 
same statement for the graph $\widetilde{\Gamma}$ constructed by subdividing
every edge $e$ of $\Gamma$ into $\bn(e)$ edges, with the trivial chain 
structure. 
We thus consider the matrix $M$ indexed by $V(\widetilde{\Gamma})$, with
$(v,v)$ entry given by the negative of the valence of $v$, and for 
$v \neq v'$, with $(v,v')$ entry given by the number of edges of 
$\widetilde{\Gamma}$ connecting $v$ to $v'$. We wish to see that the vector 
$(1,\dots,1)$ generates the kernel of $M$. If we consider $\epsilon M + I$, 
with $1/\epsilon$ at least the maximal valence in $\widetilde{\Gamma}$, we 
have a symmetric doubly stochastic matrix with nonnegative entries, which 
is irreducible because $\Gamma$ is connected.
The Perron-Frobenius theorem then implies
that the maximal eigenvalue is $1$, and is simple, which implies that 
the eigenvalue $0$ of $M$ is likewise simple, as desired.
\end{proof}

We now move on to definitions which involve geometry more directly.

\begin{sit}\label{sit:basic}
Let $X_0$ be a projective nodal curve, with dual graph
$\Gamma$, and choose an orientation on $\Gamma$. For $v \in V(\Gamma)$, 
let $Z_v$ be the corresponding irreducible component of $X_0$, and $Z_v^c$ 
the closure of the complement of $Z_v$ in $X_0$.
\end{sit}

A preliminary definition (see also Maino \cite{ma5}) is the following. 

\begin{defn}\label{def:enriched} If $X_0$ is a nodal curve with dual
graph $\Gamma$, an
\textbf{enriched structure} on $X_0$ consists of the data, for each 
$v \in V(\Gamma)$ of a line bundle $\sO_v$ on $X_0$,
satisfying the following conditions:
\begin{Ilist}
\itm for any $v \in V(\Gamma)$, we have
$$\sO_v|_{Z_v} \cong \sO_{Z_v}(-(Z_v^c \cap Z_v)),\text{ and }
\sO_v|_{Z_v^c} \cong \sO_{Z_v^c}(Z_v^c \cap Z_v);$$
\itm we have
$$\bigotimes_{v \in V(\Gamma)} \sO_v \cong \sO_{X_0}.$$
\end{Ilist}
\end{defn}

Note that it follows from the definitions that each $\sO_v$ has degree $0$.
Enriched structures always exist; they amount to suitable gluing choices at 
the nodes, and they are unique when $X_0$ is of compact type.
However,
an enriched structure is always induced by any regular smoothing of $X_0$; 
see Proposition \ref{prop:enriched-detd}.

We now explicitly introduce the chains of rational curves induced by a 
chain structure on $X_0$. 

\begin{defn}
Given $X_0$ and a chain structure $\bn$, let $\widetilde{X}_0$ denote 
the nodal curve obtained from $X_0$ by, for each $e \in E(\Gamma)$,
inserting a chain of $\bn(e)-1$ projective lines at the corresponding
node. Let $\widetilde{\Gamma}$ be the dual graph of $\widetilde{X}_0$,
with a natural inclusion $V(\Gamma) \subseteq V(\widetilde{\Gamma})$.
We refer to the new components of $\widetilde{X}_0$ as the 
\textbf{exceptional components}.
\end{defn}

Note that the above $\widetilde{\Gamma}$ is compatible with that of
Remark \ref{rem:twists}.

\begin{defn}\label{def:multideg-lb} Using our orientation of $E(\Gamma)$,
an admissible multidegree $w$ of total degree $d$ on $(X_0,\bn)$ gives a 
multidegree of total degree $d$ on $\widetilde{X}_0$
by assigning, for each $e \in E(\Gamma)$, 
degree $0$ on each component of the corresponding chain of projective curves,
except for degree $1$ on the $\mu(e)$th component when $\mu(e) \neq 0$. 
\end{defn}

The reason for restricting to such multidegrees is that extensions of
line bundles may always be chosen to have such degrees; see
Corollary \ref{cor:specialize}.

From now on, we will assume we have fixed an enriched structure together
with suitable global sections, as follows.

\begin{sit}\label{sit:enriched} In Situation \ref{sit:basic}, suppose
we have also a chain structure $\bn$ on $\Gamma$, and an enriched 
structure $(\sO_v)_v$ on the resulting $\widetilde{X}_0$, and for
each $v \in V(\widetilde{\Gamma})$, fix 
$s_v \in \Gamma(\widetilde{X}_0,\sO_v)$ vanishing precisely on $Z_v$.
\end{sit}

The sections $s_v$ will be convenient in describing maps between different 
twists of line bundles; they will not be unique even for curves of compact 
type, but in our case they are just a formal convenience, and do 
not ultimately affect our definition of limit linear series. See
Remark \ref{rem:enriched} for further discussion. 

We next describe how, given an enriched structure on $\widetilde{X}_0$, and 
a line bundle $\sL$ of multidegree $w_0$, we get a collection of line bundles
indexed by $V(G(w_0))$, with morphisms between them indexed by $E(G(w_0))$.

\begin{notn}\label{not:twist} In Situation \ref{sit:enriched} assume we
are given also an admissible multidegree $w_0$ on $(\Gamma,\bn)$. Then
for any edge $\e \in E(G(w_0))$, starting at 
$w=(w_{\Gamma},(\mu(e))_{e \in E(\Gamma)})$ and determined by twisting at
$v \in V(\Gamma)$, we have the corresponding twisting line bundle 
$\sO_{\e}$ on $\widetilde{X}_0$ defined as
$$\sO_{\e}=\sO_v \otimes \bigotimes_{e \in E(\Gamma)} 
\bigotimes_{i=1}^{\sigma(e,v)\mu(e)} \sO_{v_{e,i}},$$
where the first product is over edges $e$ adjacent to $v$, and for any
such pair, $v_{e,i}$ denotes the $i$th rational curve in 
$\widetilde{X}_0$ from $Z_v$ on the chain corresponding to $e$.

In addition, we have the section $s_{\e}$ of $\sO_{\e}$ obtained from
the tensor product of the relevant sections $s_v$ and $s_{v_{e,i}}$.

Similarly, given $w,w' \in V(G(w_0))$, let $P=(\e_1,\dots,\e_m)$ be
a minimal path from $w$ to $w'$ in $G(w_0)$, and set
$$\sO_{w,w'}=\bigotimes_{i=1}^m \sO_{\e_i}.$$
\end{notn}

In Notation \ref{not:twist}, if $\mu(e)=0$, the product over
$i$ is empty for the given $e$, and we take the
representative of $\sigma(e,v)\mu(e)$ between $0$ and $\bn(e)-1$.
Note that it follows from Proposition \ref{prop:paths-unique} that
the constructions of Notation \ref{not:twist} are independent of choices
of (minimal) paths. The reason for the notation $\sO_{w,w'}$ is that,
as one can easily verify, tensoring by $\sO_{w,w'}$ take a line bundle of
multidegree $w$ to one of multidegree $w'$.

\begin{notn}\label{not:more-twist}
In Situation \ref{sit:enriched}, suppose $\sL$ is a line bundle on
$\widetilde{X}_0$ of multidegree $w_0$. Then for any $w \in V(G(w_0))$, set
$$\sL_w := \sL \otimes \sO_{w_0,w}.$$

Given also $w_v \in V(G(w_0))$ concentrated at $v$, set
$$\sL^v:=\sL_{w_v}|_{Z_v}.$$

Given an edge $\e$ from $w$ to $w'$ in $G(w_0)$, corresponding to
twisting at $v$, then either
$\sL_{w'}=\sL_w \otimes \sO_{\e}$, or 
$\sL_w = \sL_{w'} \otimes \sO_{w',w}$. In the former case, we get a morphism
$\sL_w \to \sL_{w'}$ induced by $s_{\e}$. In the latter case, we observe
that $\sO_{w',w} \otimes \sO_{\e} \cong \sO_{\widetilde{X}_0}$, and
fixing such an isomorphism and again using $s_{\e}$ gives an induced 
morphism
$$\sL_w \to \sL_w \otimes \sO_{\e} = 
\sL_{w'}\otimes \sO_{w',w} \otimes \sO_{\e} \cong \sL_{w'}.$$
In either case, pushing forward gives an induced morphism
$$f_{\e}:\Gamma(\widetilde{X}_0, \sL_w) \to 
\Gamma(\widetilde{X}_0, \sL_{w'}).$$

Finally, if $P=(\e_1,\dots,\e_m)$ is any path in $G(w_0)$, set
$$f_P:=f_{\e_m} \circ \cdots \circ f_{\e_1}.$$
If $P$ is a minimal path from $w$ to $w'$, write
$$f_{w,w'}:=f_P.$$
\end{notn}

We have the following simple consequence of Proposition 
\ref{prop:paths-unique}:

\begin{cor}\label{cor:maps-defined} For any $w,w' \in V(G(w_0))$, the
morphism $f_{w,w'}$ is independent of the choice of minimal path.
\end{cor}

We can now give the definition of a limit linear series. As mentioned 
previously, the idea is simply that a collection of $\fg^r_{d_v}$s on
the components $Z_v$ of $X_0$ should constitute a limit linear series
precisely when it is possible to use them to glue together an 
$(r+1)$-dimensional space of sections on all of $X_0$ in any desired 
multidegree.

\begin{defn}\label{def:lls}
Let $X_0$ be a projective nodal curve, $\bn$ a chain structure,
$w_0$ an admissible multidegree of total degree $d$ on $(X_0,\bn)$, and
$(\sO_v)_{v \in V(\Gamma)}$ an enriched structure on $\widetilde{X}_0$. 
Choose also a tuple $(w_v)_{v \in V(\Gamma)}$ of vertices of $G(w_0)$, with
each $w_v$ concentrated at $v$, and sections $(s_v)_v$ as in Situation 
\ref{sit:enriched}.  Then a \textbf{limit linear series}
on $(X_0,\bn)$ consists of a line bundle $\sL$ of multidegree $w_0$ on
$\widetilde{X}_0$, together with subspaces $V^v$ of $\Gamma(Z_v,\sL^v)$ for
each $v \in V(\Gamma)$, satisfying the condition that for all 
$w \in V(G(w_0))$, the natural morphism
\begin{equation}\label{eq:gluing-map-1} 
\Gamma(\widetilde{X}_0, \sL_w) 
\to \bigoplus_{v \in V(\Gamma)} \Gamma(Z_v, \sL^{v})/V^v
\end{equation}
has kernel of dimension at least $r+1$, where \eqref{eq:gluing-map-1} is
obtained as the composition
\begin{multline*}
\Gamma(\widetilde{X}_0, \sL_w) \overset{\oplus f_{w,w_v}}{\to} 
\bigoplus_{v \in V(\Gamma)} \Gamma(\widetilde{X}_0, \sL_{w_v})
\\ \to \bigoplus_{v \in V(\Gamma)} \Gamma(Z_v, \sL^{v})
\to \bigoplus_{v \in V(\Gamma)} \Gamma(Z_v, \sL^{v})/V^v.
\end{multline*}
\end{defn}

Clearly, the choices of concentrated multidegrees are necessary to
even define the data underlying a limit linear series. However, we will
show in Proposition \ref{prop:indep-of-w} below that the resulting 
moduli space of limit linear series does not depend on this choice.

\begin{rem}\label{rem:enriched} Even in the compact type case, the
sections $s_v$ of Situation \ref{sit:enriched} are not typically unique,
even up to scaling: indeed, if $v$ disconnects $\Gamma$, then $s_v$ can
be scaled independently on (the subcurves corresponding to) each
resulting connected component. Thus, \textit{a priori} our definition 
of limit linear series depends on extra data even in the compact type 
case. However, the choice of $s_v$ is unique up to scaling on each
component, and because the maps \eqref{eq:gluing-map-1} are obtained
by restricting to individual components, their kernels do not depend on
the choice of $s_v$. Consequently, we see that the notion of limit linear 
series is in fact independent of the choices of the $s_v$. This is 
different from the notion of linked linear series introduced in \cite{os20}
(generalizing \cite{os8}), where even for curves of compact type, the
choice of $s_v$ does have an effect.
\end{rem}

\begin{rem}\label{rem:self-nodes}
We have chosen not to allow self-nodes (i.e., nodes on single irreducible
components) not because they are harder to handle, but because they are 
already better understood, and our techniques don't add anything new for
them. It is not difficult to combine our techniques with those developed 
for self-nodes, but we have chosen to present our definitions and results
without any self-nodes because we would have to systematically treat the
two types of nodes differently. 

If one wants to treat limit linear series
on a reducible curve with some self-nodes, there are several options:
the first, which is simplest to state but probably least effective for
computation is to simply introduce new rational components at each
self-node, thereby removing all self-nodes; another option is to work with
sheaves which are allowed to be torsion-free (but not invertible) at the 
self-nodes. In the latter case, one can study the resulting linear series 
by partially 
normalizing at the self-nodes and studying linear series on the resulting
smooth component(s), imposing a secancy condition at each
pair of points lying above self-nodes at which the sheaf was invertible 
(above nodes at which the sheaf was not invertible, one does not
have a gluing condition, but the degree on the relevant component is
decreased). This approach was developed already by Kleiman \cite{kl3}
nearly 40 years ago.
\end{rem}

\begin{rem}\label{rem:v-reduced} Obviously, concentrated multidegrees are
not unique, so a choice of these is a necessary input to our definition
of limit linear series. Although Proposition \ref{prop:indep-of-w} asserts 
that in fact the resulting limit linear series moduli spaces will not depend
on the choice of the $w_v$, it is still natural to wonder to what extent
one can make canonical choices of the tuples $(w_v)_v$ of concentrated 
multidegrees. The answer likely comes from the theory of $v$-reduced 
divisors, which plays an important role in Brill-Noether theory for graphs. 
However, since our theory goes through fully
as long as the $w_v$ are concentrated, it seems potentially advantageous 
not to place any further restrictions on them. Thus, the question of
canonical choices is
rather orthogonal to the purpose of the present paper, and for the sake
of simplicity we do not pursue it here. 
\end{rem}

\section{Families and moduli schemes}

In this section, we construct a moduli scheme of limit linear series,
show that it is independent of the choice of tuple of concentrated
multidegree, and finally give an alternate description which generalizes 
to the case of smoothing families. The main technical tool is the generalized 
determinantal loci introduced in Appendix B of \cite{os20}, and the
main issue that needs to be addressed for smoothing families is that
the limit linear series are defined in terms of (sections of) line bundles 
$\sL^v$ on individual components of the reducible curve, which no longer 
makes sense in a smoothing family. This difficulty is resolved by working 
instead with the line bundles $\sL_{w_v}$ on the whole curve, with
multidegree concentrated on the relevant component. 

This section is of a foundational nature, and later sections are
largely independent from it, with the exception of Theorem
\ref{thm:smoothing}, our smoothing result. Note, however, that our
specialization result, Corollary \ref{cor:specialize}, is proved in
this section.

First, we set the following notation.

\begin{notn}\label{not:prw}
In the situation of Definition \ref{def:lls},
let $P^r_{w_{\bullet}}(X_0,\bn,(\sO_v)_v)$
be the scheme parametrizing tuples
$(\sL,(V^v)_{v \in V(\Gamma)})$, where $\sL$ is a line bundle
on $\widetilde{X}_0$ of multidegree $w_0$, and each $V^v$ is an
$(r+1)$-dimensional space of global sections of the induced
line bundle $\sL^v$ on $Z_v$. 
\end{notn}

Thus, if we write $d_v=\deg \sL^v$, then 
$P^r_{w_{\bullet}}(X_0,\bn,(\sO_v)_v)$ can naturally be constructed as
a fibered product of $\Pic^{w_0}(\widetilde{X}_0)$ with the spaces 
$G^r_{d_v}(Z_v)$, fibered over the spaces $\Pic^{d_v}(Z_v)$.

We then construct a moduli scheme of limit linear series as follows.

\begin{defn}\label{def:lls-space} In the situation of Definition 
\ref{def:lls}, write $\sM$ for the universal line bundle on
$P^r_{w_{\bullet}}(X_0,\bn,(\sO_v)_v) \times X_0$, and $\sV^v$ for the 
universal subbundles of the induced $p_{1*} \sM^v$. Then let
$G^r_{\bar{w}_0}(X_0,\bn,(\sO_v)_v)$ be the closed
subscheme of $P^r_{w_{\bullet}}(X_0,\bn,(\sO_v)_v)$ defined by the intersection
over $w \in V(G(w_0))$ of the $(r+1)$st vanishing loci of the maps
\begin{equation}\label{eq:lls-map-1}
p_{1*} \sM_w \to \bigoplus_{v \in V(\Gamma)} (p_{1*} \sM^v)/\sV^v.
\end{equation}
\end{defn}

In the above, the $(r+1)$st vanishing locus is a canonical scheme structure 
on the set of points on which the kernel has dimension at least $r+1$,
defined in Appendix B of \cite{os20}. Thus, 
$G^r_{\bar{w}_{0}}(X_0,\bn,(\sO_v)_v)$ 
is a canonical scheme structure on the set of limit linear series described 
in Definition \ref{def:lls}. The notation $\bar{w}_0$ represents the
collection of admissible multidegrees obtained from $w_0$ by twisting 
(that is, $V(G(w_0))$); we use it because we will prove shortly, in 
Proposition \ref{prop:indep-of-w}, that the choice of the $w_v$ does not 
affect the resulting moduli scheme.

The $T$-valued points of 
$P^r_{w_{\bullet}}(X_0,\bn,(\sO_v)_v)$ are tuples $(\sL,(V^v))$, where
$\sL$ is a line bundle on $T \times \widetilde{X}_0$ of multidegree $w_0$, 
and each $V^v$ is a rank-$(r+1)$ subbundle of $p_{1*} \sL^v$ (in the sense
of Definition 4.2 of \cite{os8}). 
Such a tuple is a $T$-valued point of $G^r_{\bar{w}_{0}}(X_0,\bn,(\sO_v)_v)$ 
if for all $w \in V(G(w_0))$, the map
\begin{equation}\label{eq:t-vald-1}
p_{1*} \sL_w \to \bigoplus_{v \in V(\Gamma)} (p_{1*} \sL^v)/V^v
\end{equation}
has $(r+1)$st vanishing locus equal to all of $T$.

Our next task is to show that in fact, for a fixed $w_0$, the spaces of
limit linear series for different choices of the $w_v$ are canonically 
identified with one another.

A preliminary fact is the following.

\begin{prop}\label{prop:concentrated-inject} Let $\sL$ be a line bundle
of multidegree $w \in V(G(w_0))$ on $\widetilde{X}_0$, and suppose that 
$w$ is concentrated at $v$. Then the restriction map
$$H^0(\widetilde{X}_0,\sL) \to H^0(Z_v,\sL|_{Z_v})$$
is injective.
\end{prop}

\begin{proof} The main point is that for any vertices $v',v''$, and any
section $s \in H^0(\widetilde{X}_0,\sL)$ which vanishes on $Z_{v''}$, then 
the number of zeroes consequently imposed on $Z_{v'}$ is equal to the
change in index $v'$ when we take the negative twist of $w$ at $v''$.
Indeed, if there are $m$ nodes of $X_0$ connecting $Z_{v'}$ to $Z_{v''}$ 
for which $\sL$ is trivial on the associated exceptional chain 
(equivalently, for which $\mu(e)=0$), then $s|_{Z_{v'}}$ must vanish at 
these $m$ nodes, but $m$ is also the amount by which the negative twist of 
$w$ at $v''$ reduces the degree at $v'$. Given this, if $s$ vanishes on
$v$, then we simply traverse $\Gamma$ in the ordering provided by the
definition of concentration, and vanishing on $Z_{v'}$ for all the previous 
vertices $v'$ implies vanishing at the next component as well.
\end{proof}

\begin{cor}\label{cor:concentrated-subbundle} Suppose that $w$ is 
concentrated at $v$, that $w_v$ can be obtained from $w$ by twisting at 
vertices other than $v$, and that we have a $T$-valued tuple
$(\sL,(V^v)_v)$ such that the $(r+1)$st vanishing locus of
\eqref{eq:t-vald-1} is all of $T$. Then the kernel of \eqref{eq:t-vald-1} 
is a subbundle of $p_{1*} \sL_w$
of rank $r+1$, and is equal to the preimage of $V^v$ under the map
$p_{1*} \sL_w \to p_{1*} \sL^v$. Moreover, both statements hold after
arbitrary base change.
\end{cor}

\begin{proof} First observe that the hypotheses on $w$, together with
Proposition \ref{prop:concentrated-inject}, imply that the map 
$$p_{1*} \sL_w \to p_{1*} \sL_w|_{Z_v} \to p_{1*} \sL^v$$
is injective on points, and hence universally injective.
Now, by hypothesis the $(r+1)$st vanishing locus of 
\eqref{eq:t-vald-1} is all of $T$. On the other hand, at any point,
the kernel is contained in the preimage of (the corresponding fiber of) 
$V^v$, which has dimension $r+1$ by the above injectivity. 
Thus, the $(r+2)$nd vanishing locus of \eqref{eq:t-vald-1} is empty.
The statement of the corollary then follows
from Proposition B.3.4 and Lemma B.2.3 (iv) of \cite{os20}.
\end{proof}

\begin{prop}\label{prop:indep-of-w} In the situation of Definition
\ref{def:lls}, let $(w'_v)_v$ be another choice of admissible 
multidegrees concentrated at the $v \in V(\Gamma)$. Then the schemes
$G^r_{\bar{w}_{0}}(X_0,\bn,(\sO_v)_v)$ obtained from $(w_v)_v$ and
$(w'_v)_v$ are canonically isomorphic.
\end{prop}

\begin{proof} It is clearly enough to treat the case that $w'_{v'}=w_{v'}$
for all $v'$ other than some fixed choice of $v$. We first observe that
it follows from Proposition \ref{prop:paths-unique} that
given any $w,w' \in V(G(w_0))$, there is some $w''$ such that the minimal
paths from $w$ to $w''$ and from $w'$ to $w''$ do not require twisting at
$v$. Indeed, if we take a minimal path $P$ from $w$ to $w'$, and let $w''$
be obtained by all twists in $P$ except those at $v$, then the minimal path
from $w''$ to $w'$ involves only twists at $v$, so the minimal path from $w'$
to $w''$ does not require twisting at $v$. Moreover, if $w'$ is concentrated
at $v$, then we see from the construction that $w''$ is also concentrated at
$v$.
Thus, to prove the proposition we may further assume
that $w'_v$ is obtained from $w_v$ by twisting at vertices other than
$v$.
In particular, if, for a given $\sL$ of multidegree $w_0$, we let
$\sL^v$ be as usual, and $\sL'^v$ the corresponding line bundle obtained
from $w'_v$, then the map $\sL^v \to \sL'^v$ is (universally) injective.
It follows that a subbundle $V^v$ of $\sL^v$ induces a subbundle $V'^v$ of
$\sL'^v$,
so we obtain a morphism
$$P^r_{w_{\bullet}}(X_0,\bn,(\sO_v)_v) \to
P^r_{w'_{\bullet}}(X_0,\bn,(\sO_v)_v)$$ 
which we wish to show is an isomorphism on the closed subschemes of 
limit linear series.

First, if $(\sL,(V^v)_v)$ is a $T$-valued point of
$G^r_{\bar{w}_{0}}(X_0,\bn,(\sO_v)_v) \subseteq 
P^r_{w_{\bullet}}(X_0,\bn,(\sO_v)_v)$, and $w \in V(G(w_0))$, we need to
check that the $(r+1)$st vanishing locus of 
\begin{equation}\label{eq:lls-map-1'}
p_{1*} \sL_w \to p_{1*} \sL'^v/V'^v \oplus \bigoplus_{v' \neq v} 
(p_{1*} \sL^{v'})/V^{v'}
\end{equation}
is all of $T$. But by construction, $p_{1*} \sL^v/V^v$ injects into 
$p_{1*} \sL'^v/V'^v$ (universally),
so if there is a minimal path from $w$ to $w'_v$ factoring through
$w_v$, then the kernel of \eqref{eq:lls-map-1'} is identified with that of 
\eqref{eq:t-vald-1}, so the hypothesis that $(\sL,(V^v)_v)$ is in
$G^r_{\bar{w}_{0}}(X_0,\bn,(\sO_v)_v)$ together with Proposition B.3.2 of
\cite{os20} implies that the $(r+1)$st vanishing locus of 
\eqref{eq:lls-map-1'} is all of $T$. Otherwise, we have that the
composition of minimal paths from $w$ to $w_v$ and from $w_v$ to $w'_v$
is not minimal, meaning that it includes a twist at $v$; since the latter
does not have such a twist, we conclude that the minimal path from $w$ to
$w_v$ includes a twist at $v$. In this case, the minimal path from $w_v$
to $w$ does not include a twist at $v$. Let $V_v$ denote the kernel of
\eqref{eq:t-vald-1} in multidegree $w_v$; by Corollary 
\ref{cor:concentrated-subbundle} this is a subbundle which is equal to
the preimage of $V^v$.
Because the minimal path from $w_v$ to $w$ does not include a twist at $v$,
and $w_v$ is concentrated at $v$, we see that the map
$p_{1*} \sL_{w_v} \to p_{1*} \sL_w$ is universally injective, so the
image of $V_v$ is a subbundle of rank $r+1$, which is easily verified to
be in the kernel of \eqref{eq:lls-map-1'}, since $V_v$ is in the kernel of
\eqref{eq:t-vald-1}.
We conclude from Proposition B.3.4 of \cite{os20} that the $(r+1)$st
vanishing locus is all of $T$, as desired.

Now, suppose that $(\sL,(V'^v)_v)$ is a $T$-valued point of
$G^r_{\bar{w}_0}(X_0,\bn,(\sO_v)_v) \subseteq
P^r_{w'_{\bullet}}(X_0,\bn,(\sO_v)_v)$. In order to lift to 
$P^r_{w_{\bullet}}(X_0,\bn,(\sO_v)_v)$, we will set $V^{v'}=V'^{v'}$ for 
all $v' \neq v$. At $v$, we consider \eqref{eq:lls-map-1'} for $w=w_v$,
and apply Corollary \ref{cor:concentrated-subbundle} again to conclude
that the kernel of \eqref{eq:lls-map-1'} is a subbundle of rank $r+1$
which is equal to the preimage of $V'^v$ under the universal injection
$$p_{1*} \sL_{w_v} \hookrightarrow p_{1*} \sL^v \hookrightarrow 
p_{1*} \sL'^v.$$
Put differently, $V'^v$ must be contained in (the image of) 
$p_{1*} \sL_{w_v}$. Then set $V^v$ to be the preimage of $V'^v$ in 
$p_{1*} \sL^v$, or 
equivalently, the image of the kernel of \eqref{eq:lls-map-1'}. This gives
a ($T$-valued) point of $P^r_{w_{\bullet}}(X_0,\bn,(\sO_v)_v)$ mapping to 
$(\sL,(V'^v)_v)$, and it is clear from the above injectivities that such
a point is unique. It thus remains to check that the point we have 
constructed lies in $G^r_{\bar{w}_0}(X_0,\bn,(\sO_v)_v)$.

Given any $w$, we know that the kernel of \eqref{eq:lls-map-1'} has 
$(r+1)$st vanishing locus equal to $T$, and we wish to verify the same
for the kernel of \eqref{eq:t-vald-1}. If there is a minimal path from $w$
to $w'_v$ factoring through $w_v$, then we are in the same situation as
above, and we get the desired statement. On the other hand, if the
minimal path from $w$ to $w_v$ includes a twist at $v$, 
then in \eqref{eq:t-vald-1}
the map to the summand $(p_{1*} \sL^v)/V^v$ is zero, so we conclude that
\eqref{eq:t-vald-1} factors through \eqref{eq:lls-map-1'}, and then by
Corollary B.3.5 of \cite{os20} it follows that the $(r+1)$st vanishing locus 
of \eqref{eq:t-vald-1} is all of $T$. The proposition follows.
\end{proof}

We now describe a second version of the moduli space construction, which 
is less immediately related to our definition of limit linear series, but
which works transparently in families of curves; we will then show in 
Proposition \ref{prop:first-compare} that on the special fiber, the two 
constructions are canonically isomorphic.

\begin{notn}\label{not:prw-2} In the situation of Definition \ref{def:lls}, 
let $\widetilde{P}^r_{w_{\bullet}}(X_0,\bn,(\sO_v)_v)$
be the scheme parametrizing tuples
$(\sL,(V_v)_{v \in V(\Gamma)})$, where $\sL$ is a line bundle
on $\widetilde{X}_0$ of multidegree $w_0$, and each $V_v$ is an
$(r+1)$-dimensional space of global sections of the induced
line bundle $\sL_{w_v}$ on $\widetilde{X}_0$. 
\end{notn}

Denote by
$G^r_{w}(\widetilde{X}_0)$ the moduli scheme of pairs $(\sL,V)$, where $\sL$ 
has multidegree $w$ on $\widetilde{X}_0$,
and $V$ is an $(r+1)$-dimensional 
space of global sections of $\sL$. 
We thus have that $\widetilde{P}^r_{w_{\bullet}}(X_0,\bn,(\sO_v)_v)$ can naturally be 
constructed as the product over $v \in V(\Gamma)$ of the spaces 
$G^r_{w_v}(\widetilde{X}_0)$, fibered over $\Pic^{w_0}(\widetilde{X}_0)$
via twisting by $\sO_{w_v,w_0}$.
In particular, $\widetilde{P}^r_{w_{\bullet}}(X_0,\bn,(\sO_v)_v)$ is
proper over $\Pic^{w_0}(\widetilde{X}_0)$.

\begin{defn}\label{def:lls-space-2} In the situation of Notation
\ref{not:prw-2}, let $\widetilde{\sM}$ be the universal line bundle on
$\widetilde{P}^r_{w_{\bullet}}(X_0,\bn,(\sO_v)_v) \times X_0$, and for each
$w \in V(G(w_0))$, let $\widetilde{\sM}_w$ be induced by twisting as
before. Then for each $v \in V(\Gamma)$, let $\widetilde{\sV}_v$ be the 
universal subbundles of $p_{1*} \widetilde{\sM}_{w_v}$, and let
$\widetilde{G}^r_{\bar{w}_{0}}(X/B,\bn,(\sO_v)_v)$ be the closed
subscheme of $\widetilde{P}^r_{w_{\bullet}}(X/B,\bn,(\sO_v)_v)$ defined by the 
intersection over $w \in V(G(w_0))$ of the $(r+1)$st vanishing loci of the maps
\begin{equation}\label{eq:lls-map-2}
p_{1*} \widetilde{\sM}_w \to 
\bigoplus_{v \in V(\Gamma)} (p_{1*} \widetilde{\sM}_{w_v})/\widetilde{\sV}_v.
\end{equation}
\end{defn}

Thus, a $T$-valued point of 
$\widetilde{P}^r_{w_{\bullet}}(X_0,\bn,(\sO_v)_v)$ is a tuple $(\sL,(V_v))$, 
where $\sL$ is a line bundle on $T \times \widetilde{X}_0$ of multidegree 
$w_0$, and each $V_v$ is a rank-$(r+1)$ subbundle of $p_{1*} \sL_{w_v}$. 
Such a tuple is a $T$-valued point of 
$\widetilde{G}^r_{\bar{w}_{0}}(X_0,\bn,(\sO_v)_v)$ 
if for all $w \in V(G(w_0))$, the map
\begin{equation}\label{eq:t-vald-2}
p_{1*} \sL_w \to \bigoplus_{v \in V(\Gamma)} (p_{1*} \sL_{w_v})/V_v
\end{equation}
has $(r+1)$st vanishing locus equal to all of $T$.

We now check that our two constructions are equivalent. Note that it
follows in particular that
$\widetilde{G}^r_{\bar{w}_{0}}(X_0,\bn,(\sO_v)_v)$ is also independent of
the choice of $(w_v)_v$.

\begin{prop}\label{prop:first-compare} In the situation of Definition
\ref{def:lls}, restriction to the components $Z_v$ induces an isomorphism
$$\widetilde{G}^r_{\bar{w}_{0}}(X_0,\bn,(\sO_v)_v) \risom G^r_{\bar{w}_0}(X_0,\bn,(\sO_v)_v).$$ 
\end{prop}

\begin{proof} 
We first verify that restriction to the $Z_v$ induces a morphism
\begin{equation}\label{eq:prw}
\widetilde{P}^r_{w_{\bullet}}(X_0,\bn,(\sO_v)_v) 
\to P^r_{w_{\bullet}}(X_0,\bn,(\sO_v)_v),
\end{equation}
which amounts to the assertion that if $V_v$ is a subbundle of
$p_{1*} \sL_{w_v}$ on some scheme $T$ over $\Spec k$, then restricting 
$V_v$ to $Z_v$ induces a subbundle of the same rank of $p_{1*} \sL^v$. By 
Lemma B.2.3 (iii) of \cite{os20}, this follows from injectivity of
restriction on points, 
which is Proposition \ref{prop:concentrated-inject}.
Next, that \eqref{eq:prw} induces a morphism
$$\widetilde{G}^r_{\bar{w}_{0}}(X_0,\bn,(\sO_v)_v) \to 
G^r_{\bar{w}_0}(X_0,\bn,(\sO_v)_v)$$
is immediate from the fact that \eqref{eq:t-vald-1} factors through
\eqref{eq:t-vald-2}, using Corollary B.3.5 of \cite{os20}. 

It thus remains to prove that this morphism is an isomorphism, or 
equivalently that every $T$-valued point of 
$G^r_{\bar{w}_0}(X_0,\bn,(\sO_v)_v)$ lifts to a unique point of
$\widetilde{G}^r_{\bar{w}_{0}}(X_0,\bn,(\sO_v)_v)$. Accordingly, suppose
that $(\sL,(V^v)_{v \in V(\Gamma)})$ is a $T$-valued point of
$G^r_{\bar{w}_0}(X_0,\bn,(\sO_v)_v)$; by the injectivity of the 
maps $p_{1*} \sL_{w_v} \to p_{1*} \sL^v$,
a lift $(\sL,(V_v)_{v \in V(\Gamma)})$ is unique, if it exists.
Next, for any $v \in V(\Gamma)$, if we consider the multidegree $w_v$, 
Corollary \ref{cor:concentrated-subbundle} implies that the kernel
of \eqref{eq:t-vald-1} is a subbundle of $p_{1*} \sL_{w_v}$ of rank 
$r+1$, which is the preimage of $V^v$. We thus set this kernel as our $V_v$.
Thus, it is enough to see that with this
choice of the bundles $V_v$, we have that for every multidegree $w$,
the $(r+1)$st vanishing locus of \eqref{eq:t-vald-2} is all of $T$. 
But by construction, for each $v$ the natural map 
$\left(p_{1*} \sL_{w_v}\right)/V_v \to
\left(p_{1*} \sL^v\right)/V^v$
is injective, even after arbitrary base change,
so it follows that for
any $w$, the kernels of \eqref{eq:t-vald-1} and \eqref{eq:t-vald-2}
are identified, likewise after arbitrary base change. Then the $(r+1)$st
vanishing loci agree by Proposition B.3.4 of \cite{os20}, giving the
desired statement.
\end{proof}

We conclude this section by explaining how the construction of 
Definition \ref{def:lls-space-2} works in families, and applying it to 
prove a specialization statement.

First, the families of curves we will consider are as follows:

\begin{defn}\label{def:family} We say that $\pi:X \to B$ is a
\textbf{smoothing family} if $B=\Spec R$ for $R$ a DVR, and further:
\begin{Ilist}
\itm $\pi$ is flat and proper;
\itm the special fiber $X_0$ of $\pi$ is a (split) nodal curve;
\itm the generic fiber $X_{\eta}$ of $\pi$ is smooth;
\itm $\pi$ admits sections through every component of $X_0$.
\end{Ilist}

If further $X$ is regular, we say that $\pi$ is a \textbf{regular} smoothing
family.
\end{defn}

See Remark \ref{rem:smoothing-fam} below for discussion of our choice
of level of generality. Condition (IV) is always satisfied after etale
base change, and is used to ensure the existence of a Picard scheme with
universal line bundle.

Associated to a smoothing family we still have a
dual graph $\Gamma$: namely, the dual graph of the special fiber $X_0$. We 
then continue to use the notation $Z_v$ to denote the component of $X_0$ 
corresponding to a vertex $v$ of $\Gamma$. In this situation, one may
define an enriched structure as before, with the additional condition
that there should exist sections $s_v$ as in Situation \ref{sit:enriched}.
We then see:

\begin{prop}\label{prop:enriched-detd} If $\pi:X \to B$ is a
regular smoothing family, then an enriched structure is uniquely determined
by setting $\sO_v=\sO_X(Z_v)$, and $s_v$ as in Situation \ref{sit:enriched} 
are then induced by the canonical inclusions $\sO_X \to \sO_X(Z_v)$.
Moreover, this choice induces an enriched structure together with suitable
sections on $X_0$ via restriction.
\end{prop}

Now, we introduce the following terminology to take chain structures into
account.

\begin{defn}\label{defn:level-family} Given $(X_0,\bn)$ and a regular 
smoothing family $\widetilde{\pi}:\widetilde{X} \to \widetilde{B}$ with 
$\widetilde{B}$ the spectrum of a DVR, we say that $\widetilde{\pi}$ is 
of \textbf{fiber type} $(X_0,\bn)$ if the special fiber of 
$\widetilde{\pi}$ is isomorphic to (a base extension of) the
curve $\widetilde{X}_0$ obtained from $(X_0,\bn)$. Given also a
smoothing family $\pi:X \to B$, with special fiber $X_0$, we say that
$\widetilde{\pi}$ is an \textbf{extension} of $\pi$ if it is obtained
from $\pi$ via base extension followed by iterated blowups at the nodes
of the special fiber.
\end{defn}

Whenever we say $\pi$ is of fiber type $(X_0,\bn)$, we implicitly
assume that we have fixed an isomorphism between the special fiber of $\pi$ 
and the appropriate base extension of $\widetilde{X}_0$.

(Regular) smoothing families of type $(X_0,\bn)$ arise naturally in two 
different ways: the first is as extensions of a given regular smoothing 
family $\pi$, taken for instance in order to extend the generic point
to a field of definition of a line bundle on the geometric generic fiber,
in which case the line bundle will extend over the extended family.
The second is as regularizations of irregular families, in which case no 
base change is involved. The former will be more immediately important to
us, but our theory is general enough to handle both situations at once.

\begin{sit}\label{sit:grd-family} Suppose $\widetilde{\pi}:\widetilde{X}
\to \widetilde{B}$ is a regular smoothing family of fiber type $(X_0,\bn)$,
and we also fix an admissible multidegree $w_0$ on $\widetilde{X}_0$,
as well as a tuple $(w_v)_{v \in V(\widetilde{\Gamma})}$ of vertices of 
$G(w_0)$, with each $w_v$ concentrated at $v$. Let
$(\sO_v,s_v)_{v \in V(\Gamma)}$ be the enriched structure and associated
sections on 
$\widetilde{X}_0$ given by Proposition \ref{prop:enriched-detd}.
\end{sit}

In Situation \ref{sit:grd-family}, given $w \in V(G(w_0))$, denote by 
$\Pic^{w}(\widetilde{X}/\widetilde{B})$ the moduli schemes of line 
bundles of degree $d$
which have multidegree $w$ on fibers lying over the closed point of
$\widetilde{B}$.
Then denote by
$G^r_{w}(\widetilde{X}/\widetilde{B})$ the moduli scheme of pairs 
$(\sL,V)$, where $\sL$ is in $\Pic^w(\widetilde{X}/\widetilde{B})$, and 
$V$ is an $(r+1)$-dimensional space of global sections of $\sL$. 
The representability of these spaces is standard; one can argue just as in 
the proof of Theorem 5.3 of \cite{os8}, for instance.
The maps \eqref{eq:lls-map-2} generalize to this situation, and
we can then generalize the previous constructions to the case of families.

\begin{notn}\label{notn:grd-fam}
Construct 
$\widetilde{P}^r_{w_{\bullet}}(\widetilde{X}/\widetilde{B},X_0,\bn,(\sO_v)_v)$
as the product over $v \in V(\Gamma)$ of the spaces 
$G^r_{w_v}(\widetilde{X})$, fibered over $\Pic^{w_0}(\widetilde{X})$, 
and let
$\widetilde{G}^r_{\bar{w}_{0}}(\widetilde{X}/\widetilde{B},X_0,\bn,(\sO_v)_v)$ 
be the closed subscheme defined by the intersection of the $(r+1)$st 
vanishing loci of the maps \eqref{eq:lls-map-2}, as $w$ varies over 
$V(G(w_0))$.
\end{notn}

We then have the following basic fact.

\begin{prop}\label{prop:proper-generic} 
The moduli scheme 
$\widetilde{G}^r_{\bar{w}_{0}}(\widetilde{X}/\widetilde{B},X_0,\bn,(\sO_v)_v)$ 
is proper over $\Pic^{w_0}(\widetilde{X}/\widetilde{B})$. Its generic fiber 
is naturally identified with $G^r_d(X_{\eta})$, and its special fiber
with (the appropriate base extension of) 
$\widetilde{G}^r_{\bar{w}_{0}}(X_0,\bn,(\sO_v)_v)$.
\end{prop}

\begin{proof} The first statement is immediate from the construction,
as is the statement on the special fiber. The description of the generic
fiber follows from the observation that the maps $f_{w,w'}$ are
all isomorphisms over the generic fiber; in fact, we claim that if we fix any 
$v$, then an arbitrary choice of $V_v$ uniquely determines $V_{v'}$ as
the image of $V_v$ for all $v' \neq v$. Indeed, using Proposition B.3.4
and Lemma B.2.3 (iv) of \cite{os20}, we see that for a given choice of $V_v$,
if we consider $w=w_v$ we will have the desired condition on the $(r+1)$st
vanishing locus of \eqref{eq:t-vald-2} if and only if $V_v$ maps into each
of the $V_{v'}$,
which is the same as saying that $V_{v'}$ is the image of $V_v$.
On the other hand, if $V_{v'}$ is the image of $V_v$ for all $v'$, we see
that the kernel of \eqref{eq:t-vald-2} for any $w$ is simply the image of 
$V_v$, so we have the desired behavior of the $(r+1)$st vanishing locus. 
\end{proof}

\begin{cor}\label{cor:specialize} Let $\pi:X \to B$ be a smoothing
family, with special fiber $X_0$.
Let $(\sL,V)$ be a $\fg^r_d$ on the geometric generic fiber $X_{\bar{\eta}}$.
Then there exists a chain structure $\bn$ on $X_0$, an extension
$\widetilde{\pi}:\widetilde{X} \to \widetilde{B}$ of $\pi$ having
fiber type $(X_0,\bn)$, and an admissible
multidegree $w_0$ on the resulting $\widetilde{X}_0$ such that
$\sL$ extends to a line bundle of multidegree $w_0$ on $\widetilde{X}$.

For any such $\bn$, $\widetilde{\pi}$, and $w_0$, and any collection of 
$w_v \in V(G(w_0))$ concentrated at each $v \in V(\Gamma)$, we have that 
$(\sL,V)$ extends to a limit linear series on $\widetilde{X}$.
\end{cor}

\begin{proof} This is mostly standard, but also brief, so we include it
for the convenience of the reader. The last assertion is immediate from
Propositions \ref{prop:proper-generic} and \ref{prop:first-compare}.

For the first assertion, we necessarily have $(\sL,V)$ defined
over some finite extension $\eta'$ of $\eta$; let $B'$ be the 
corresponding integral closure of $B$, localized at a closed point.
If $X'=X \times _B B'$, then if we repeatedly blow up the non-smooth locus 
of $X'$ over $B'$ to obtain a regular total space, we obtain our $\bn$ and 
$\widetilde{X}$.
Since $\sL$ is now defined over the new generic fiber, and $\widetilde{X}$ 
is still regular, we can extend $\sL$ to all of $\widetilde{X}$. It
remains to see that the extension can be chosen to have admissible 
multidegree, but this is easily achieved by twisting first at non-exceptional
components to achieve sufficiently positive degree on each chain of 
exceptional components,
and then twisting at exceptional components first to achieve nonnegativity
on each component, and then admissibility.
\end{proof}

\begin{rem}\label{rem:level-mult} In fact, we see from the proof of
Corollary \ref{cor:specialize} that we have the following refined
statement: let $\bn'$ be the chain structure on $X_0$ obtained by
setting $\bn'(e)$ to be one greater than the number of blowups required
to make make $X$ regular at the point corresponding to $e$. 
Then the $\bn$ produced in the proof is of the form $\bn(e)=m \bn'(e)$
for all $e$, where $m$ is the ramification index of $B'$ over $B$.

Thus, the collection of chain structures we need to consider in order
to extend line bundles on the initial family are not arbitrary, but are
restricted to multiples of the ``base'' chain structure $\bn'$.
\end{rem}

\begin{rem}\label{rem:smoothing-fam} The base $B$ in Definition 
\ref{def:family} may be generalized considerably, but
this makes the conditions more complicated; compare Definitions 2.1.1 and
2.2.2 of \cite{os20}. Moreover, imposing the existence of an enriched
structure will imply that even if $B$ is higher-dimensional, the
geometry of the family all occurs in codimension $1$, so there seems to
be little reason to introduce additional technical complications.
\end{rem}

\section{The two-component case}\label{sec:two-comp}

In order to give the equivalent definition which will ultimately
generalize that of Eisenbud and Harris, the two-component case
is the simplest situation to consider. Conveniently, it is also the
base case of an induction argument for the more general situation, so
we will first develop the key comparison result for curves with two
components. In this case, we simplify our notation as follows.

\begin{sit}\label{sit:2-comp}
Let $X_0$ consist of two smooth curves $Z_1,Z_2$ glued to
one another at nodes $P_1,\dots,P_m$. Fix the orientation on $\Gamma$ with
all edges going from $Z_1$ to $Z_2$. Let $\bn$ be a chain structure,
and for $i=1,\dots,m$, write $n_i:=\bn(P_i)$. For $i=1,\dots,m$, and
$j=1,\dots,n_i-1$, let $E_{i,j}$ denote the $j$th exceptional component
of $\widetilde{X}_0$ lying over $P_i$ on $X_0$.
Fix an admissible multidegree $w_0$ on $(X_0,\bn)$, and
multidegrees $w_1,w_2 \in V(G(w_0))$ concentrated at $Z_1,Z_2$ 
respectively. Write $\mu_i:=\mu_1(P_i)$, where 
$w_1=((w_1)_{\Gamma},\mu_1(\bullet))$. Let $b$ be the number of twists at 
$Z_1$ required to get from $w_1$ to $w_2$. Identify $V(G(w_0))$ with $\ZZ$
by sending $w$ to the number of twists at $Z_1$ required to get from $w_1$ 
to $w$.
\end{sit}

We will assume throughout this section that we are in the above situation.
In this case, $G(w_0)$ is an unbounded chain, with edges going in each 
direction. We have identified $w_1$ with $0$, and $w_2$ with $b$.
Accordingly, for any line bundle $\sL$ of multidegree $w_0$ on
$\widetilde{X}_0$, for $i \in \ZZ$ we will write $\sL_i$ for the line bundle
$\sL_w$, where $w$ is obtained from $w_1$ by twisting $i$ times at $Z_1$.
As we have already done above, when convenient we will write nodes or 
components in place of the corresponding edges or vertices of the dual graph.

We then introduce the following notation as well.

\begin{notn}\label{not:twists-2-comp}
For any line bundle $\sL$ of multidegree $w_0$ on
$\widetilde{X}_0$, 
write $\sL^1:=\sL_0|_{Z_1}$, and $\sL^2:=\sL_b|_{Z_2}$.
\end{notn}

We now define sequences of effective divisors supported on the $P_i$ which
will be used to give multivanishing sequences. 

\begin{defn}\label{def:twists-div-2-comp}
Let $D^1_0,\dots,D^1_{b+1}$ be the sequence of effective divisors on $Z_1$ 
defined by $D^1_0=0$, and for $i\geq 0$,
$$D^1_{i+1}-D^1_{i}=\sum_{j: \mu_j +i \equiv 0 \pmod{n_j}} P_j,$$
and similarly define $D^2_0,\dots,D^2_{b+1}$ on $Z_2$ by $D^2_0=0$, and for
$i \geq 0$,
$$D^2_{i+1}-D^2_{i}=\sum_{j:\mu_j+b-i \equiv 0 \pmod{n_j}} P_j.$$
\end{defn}

The relationship between the twisting divisors and line bundles is given by 
the following basic proposition, whose proof is left to the reader.

\begin{prop}\label{prop:twists-basic} For $i=0,\dots,b+1$, we have
$$\sL_i|_{Z_1}=\sL^1(-D^1_i)\quad \text{ and } 
\sL_{i-1}|_{Z_2}=\sL^2(-D^2_{b+1-i}),$$
where we use equality to denote canonical isomorphism. 

We also have for all $i=0,\dots,b$ that $P_j$ is in the support of 
$D^1_{i+1}-D^1_{i}$
if and only if $\sL_{i}$ has degree $0$ on $E_{j,\ell}$ for all $\ell$,
and $P_j$ is in the support of $D^2_{i+1}-D^2_{i}$ if and only if
$\sL_{b-i}$ has degree $0$ on $E_{j,\ell}$ for all $\ell$. In particular,
if $E_i$ denotes the union over $j$ such that $P_j$ in the support of 
$D^1_{i+1}-D^1_i$ of the chains of exceptional components lying over
the $P_j$,
$\sL_i|_{E_i} \cong \sO_{E_i}$, and we thus get an induced isomorphism
$$\vp_i:\sL^1(-D^1_i)/\sL^1(-D^1_{i+1}) \risom
\sL^2(-D^2_{b-i})/\sL^2(-D^2_{b+1-i})$$
for each $i$.
\end{prop}

We think of the $\vp_i$ as being gluing maps; in the case of trivial
chain structure, the $\vp_i$ are each defined on all nodes at once,
but in general they are only defined on subsets of the nodes, which
depend on $i$. 

\begin{defn} In the situation of Definition \ref{def:multivanishing},
we say that $j$ is \textbf{critical} for $D_{\bullet}$ if 
$D_{j+1} \neq D_j$.
\end{defn}

Our main comparison result in the two-component case is as then follows:

\begin{lem}\label{lem:compare-2-comps} In Situation \ref{sit:2-comp},
fix also an enriched structure on $\widetilde{X}_0$, and sections $s_v$
as in Situation \ref{sit:enriched}.
For a given $(\sL,(V^1,V^2))$, and $i=1,2$,
denote by $a^i_0,\dots,a^i_r$ the 
multivanishing sequence of $V^i$ along the $D^i_{\bullet}$.
Then $(\sL,(V^1,V^2))$ is a limit linear series if and only if
\begin{Ilist}
\itm for $\ell=0,\dots,r$, if $a^1_{\ell}=\deg D^1_j$ with $j$ critical
for $D^1_{\bullet}$, then
\begin{equation}\label{eq:eh-genl-2}
a^2_{r-\ell} \geq \deg D^2_{b-j};
\end{equation}
\itm for $i=1,2$, there exist bases $s^i_0,\dots,s^i_r$ of the $V^i$ such that
$$\ord_{D_{\bullet}}s^i_{\ell} = a^i_{\ell} \quad \text{ for } 
\ell=0,\dots,r,$$
and for all $\ell$ with \eqref{eq:eh-genl-2} an equality, we have
$$\vp_{j}(s^1_{\ell})=s^2_{r-\ell}$$
when we consider $s^1_{\ell} \in V^1(-D^1_{j})$ and 
$s^2_{r-\ell} \in V^2(-D^2_{b-j})$, with $j$ as in (I).
\end{Ilist}
\end{lem}

\begin{rem}\label{rem:eh-sym} Although condition (I) appears asymmetric, 
in fact this is not the case; indeed, Proposition \ref{prop:twists-basic}
says that the construction of the 
$D^i_{\bullet}$ implies that $j$ is critical for $D^1_{\bullet}$ if and 
only if $b-j$ is critical for $D^2_{\bullet}$, so (I) is equivalent to 
requiring that if $a^2_{r-\ell}=\deg D^2_{b-j}$ with $b-j$ critical for
$D^2_{\bullet}$, then $a^1_{\ell} \geq \deg D^1_j$.
\end{rem}

As an intermediate step, it is convenient to consider a bounded version
of $G(w_0)$ as follows.

\begin{notn}\label{not:barG-2-comp} Let $\bar{G}(w_0)$ denote the directed
subgraph of $G(w_0)$ consisting of all vertices between $w_1$ and $w_2$ 
(inclusive), and with all edges of $G(w_0)$ connecting vertices in 
$V(\bar{G}(w_0))$.
\end{notn}

It turns out that in the definition of limit linear series, considering
multidegrees in $\bar{G}(w_0)$ suffices.

\begin{prop}\label{prop:compare-restrict} In the situation of
Lemma \ref{lem:compare-2-comps}, $(\sL,(V^1,V^2))$ is a limit linear
series if and only if \eqref{eq:gluing-map-1} has kernel of dimension
at least $r+1$ for all $w \in V(\bar{G}(w_0))$.
\end{prop}

\begin{proof} Since $V(\bar{G}(w_0)) \subseteq V(G(w_0))$, one direction is 
trivial. Conversely, suppose that \eqref{eq:gluing-map-1} has kernel of 
dimension at least $r+1$ for all $w \in V(\bar{G}(w_0))$, and let 
$w' \in V(G(w_0))$ be arbitrary; we need to show that 
\eqref{eq:gluing-map-1} also has kernel of
dimension at least $r+1$ in multidegree $w'$. Considering $w'=i$ for some
$i \in \ZZ$, there are three cases to consider:
either $0 \leq i \leq b$, or $i<0$, or $i>b$.
The first case is the same as having $w' \in V(\bar{G}(w_0))$, so there is
nothing to show. The other two cases being
symmetric, we only treat the case that $i<0$. In this case, we claim
that the kernel $W$ of \eqref{eq:gluing-map-1} in multidegree $w_1$
injects into the kernel of \eqref{eq:gluing-map-1} in multidegree $w'$
under $f_{w_1,w'}$. Indeed, it is clear that the entire image of 
$f_{w_1,w'}$ is contained in the kernel of \eqref{eq:gluing-map-1},
so it suffices to see that $f_{w_1,w'}$ is injective on $W$.
But $f_{w_1,w'}$ is induced by a map which is an inclusion on
$Z_1$, so the desired injectivity is an immediate consequence of 
Proposition \ref{prop:concentrated-inject}.
\end{proof}

Next, in $\bar{G}(w_0)$, we can reinterpret the kernel of 
\eqref{eq:gluing-map-1} as follows.

\begin{prop}\label{prop:maps-ker-compare} In the situation of
Lemma \ref{lem:compare-2-comps}, for $i=0,\dots,b$, consider the map
\begin{equation}\label{eq:gluing-map-2}
V^1(-D^1_i)\oplus V^2(-D^2_{b-i}) \to 
\sL^2(-D^2_{b-i})/\sL^2(-D^2_{b-i+1})
\end{equation}
induced by taking quotients, and applying $-\vp_i$ on the first factor.
Then our morphisms $H^0(\widetilde{X}_0,\sL_i) \to H^0(Z_j,\sL^j)$ 
for $j=1,2$ induce an isomorphism between the kernel of 
\eqref{eq:gluing-map-1} and the kernel of \eqref{eq:gluing-map-2}.
\end{prop}

\begin{proof} The image of
$H^0(\widetilde{X}_0,\sL_i)$ in $H^0(Z_1,\sL^1)$ (respectively, 
$H^0(Z_2,\sL^2)$) is contained in
$H^0(Z_1,\sL^1(-D^1_i))$ (respectively, $H^0(Z_2,\sL^2(-D^2_{b-i}))$)
by construction, so a section of $H^0(X_0,\sL_i)$ which lies in the
kernel of \eqref{eq:gluing-map-1} necessarily restricts to
$V^1(-D^1_i)$ on $Z_1$ and $V^2(-D^2_{b-i})$ on $Z_2$. That it in fact
yields an element in the kernel of \eqref{eq:gluing-map-2} is essentially
the definition of $\vp_i$. To see that the constructed map is bijective,
the main point is that given a pair
$(s_1,s_2) \in H^0(Z_1,\sL_i|_{Z_1}) \oplus H^0(Z_2,\sL_i|_{Z_2})$, 
an extension of $(s_1,s_2)$ to a global section 
$s \in H^0(\widetilde{X}_0,\sL_i)$ is
unique if it exists, and it exists if and only if $\vp_i(s_1)=s_2$,
using the identifications of Proposition \ref{prop:twists-basic}.
Indeed, the assertion is clear on the union of exceptional chains $E_i$ from 
the construction of $\vp_i$,
so it is enough to check that there is always a unique extension over the
exceptional chains not contained in the $E_i$. But if $E$ is such
a chain, then $\sL_i|_E$ has degree $1$ on exactly one irreducible 
component, and degree $0$ on the others, and it follows that $\sL_i|_E$ 
has a unique global section with arbitrary prescribed values at either
end of $E$, giving the desired assertion. The desired bijectivity follows.
\end{proof}

We can now finish our examination of the two-component case.

\begin{proof}[Proof of Lemma \ref{lem:compare-2-comps}]
First, by Propositions \ref{prop:compare-restrict} and 
\ref{prop:maps-ker-compare}, we have reduced to showing that
$(\sL,(V^1,V^2))$ satisfies (I) and (II) if and only if
\eqref{eq:gluing-map-2} has kernel of dimension at least $r+1$ for
$i=0,\dots,b$.

Now, observe that
$$\dim V^1(-D^1_i)=\#\{\ell:a^1_{\ell} \geq \deg D^1_i\}, \text{ and }
\dim V^2(-D^2_{b-i})=\#\{\ell:a^2_{\ell} \geq \deg D^2_{b-i}\}.$$
For each $i$, let $r_i$ be the rank of 
\eqref{eq:gluing-map-2}; note that $r_i=0$ unless $i$ is critical for
$D^1_{\bullet}$. Choose 
$\ell_1$ and $\ell_3$ minimal with 
$a^1_{\ell_1} \geq \deg D^1_i$ and $a^2_{\ell_3} \geq \deg D^2_{b-i}$,
and $\ell_2$ and $\ell_4$ maximal with with $a^1_{\ell_2} \leq \deg D^1_i$ 
and $a^2_{\ell_4} \leq \deg D^2_{b-i}$. Here, if $a^1_{\ell}<\deg D^1_i$ for
all $\ell$, set $\ell_1=r+1$, if $a^1_{\ell} > \deg D^1_i$ for all $\ell$,
set $\ell_2=-1$, and similarly for $\ell_3$ and $\ell_4$.
Then the kernel of \eqref{eq:gluing-map-2} has dimension equal to
$$\dim V^1(-D^1_{i})+\dim V^2(-D^2_{b-i})-r_i
=r+1-\ell_1+r+1-\ell_3-r_i,$$
so for the kernel of \eqref{eq:gluing-map-2} to have dimension at least $r+1$
is equivalent to
\begin{equation}\label{eq:kermap-ineq} 
\ell_1+\ell_3+r_i \leq r+1.
\end{equation}
In addition, we see that condition (II) of the lemma 
is equivalent to having that, for each critical $i$ for 
$D^1_{\bullet}$, the images of $V^1(-D^1_i)$ and $V^2(-D^2_{b-i})$ under 
\eqref{eq:gluing-map-2} overlap with dimension at least equal to 
\begin{equation}\label{eq:gluing} 
\#\{\ell:\ell_1\leq \ell \leq \ell_2, \text{ and } 
\ell_3 \leq r-\ell \leq \ell_4\}.
\end{equation}

Now, if we assume condition (I), we claim that for all $i$, we have
$\ell_1+\ell_3 \leq r+1$, and that if $i$ is critical for $D^1_{\bullet}$, 
we also have 
$\ell_3 \leq r-\ell_2$ and $\ell_1 \leq r -\ell_4$. For the first claim,
note that by definition $a_{\ell_1-1}^1 < \deg D^1_i$; if we let $i'$
be critical for $D^1_{\bullet}$ with $a_{\ell_1-1}^1 = \deg D^1_{i'}$, 
then (I)
implies that $a_{r+1-\ell_1}^2 \geq \deg D^2_{b-i'} \geq \deg D^2_{b-i}$,
so $\ell_3 \leq r+1-\ell_1$, giving the first claimed inequality. Next,
if $i$ is critical for $D^1_{\bullet}$, then there are two cases to 
consider: if $\deg D^1_i$
does not occur in $a^1$, we will have $\ell_2=\ell_1-1$, and in this
case the inequality $\ell_3 \leq r-\ell_2$ is the same as 
$\ell_3 \leq r+1-\ell_1$, which we have just proved. On the other hand, if
$\deg D^1_i$ does occur in $a^1$, then (I) gives 
$a^2_{r-\ell_2} \geq D^2_{b-i}$, which means $\ell_3 \leq r-\ell_2$,
as desired. The proof of the last claimed inequality $\ell_1 \leq r-\ell_4$ 
is similar, taking into account Remark \ref{rem:eh-sym}.

Still assuming (I), we next claim that \eqref{eq:gluing-map-2} having
kernel of dimension at least $r+1$ for $i=0,\dots,b$ is equivalent to 
condition (II).  If $i$ is not critical for $D^1_{\bullet}$,
then $r_i=0$, so we see from \eqref{eq:kermap-ineq} that the desired
kernel size follows
from the inequality $\ell_1+\ell_3 \leq r+1$, which we have already proved. 
If $i$ is critical for $D^1_{\bullet}$, using 
$\ell_3 \leq r-\ell_2$ and $\ell_1 \leq r -\ell_4$.
the inequalities in \eqref{eq:gluing} 
simplify to $r-\ell_4 \leq \ell \leq \ell_2$. Thus, the existence of the 
desired basis is equivalent to requiring that the images of
$V^1(-D^1_i)$ and $V^2(-D^2_{b-i})$ under \eqref{eq:gluing-map-2} overlap with 
dimension at least equal to $\ell_2+1-(r-\ell_4)$. On the other hand, the
dimension of this overlap is equal to the sum of the dimensions of the
images of $V^1(-D^1_i)$ and $V^2(-D^2_{b-i})$, minus $r_i$, which is to
say, $\ell_2+1-\ell_1+\ell_4+1-\ell_3-r_i$, so we conclude that (again,
assuming (I)), condition (II)
is equivalent to the inequality
$$\ell_2+1-\ell_1+\ell_4+1-\ell_3-r_i \geq \ell_2+1-(r-\ell_4),$$
which is the same as \eqref{eq:kermap-ineq}. This proves the claim,
and we conclude that (I) and (II) together imply that $(\sL,(V^1,V^2))$
is a linear linear series, and moreover, that to see the converse, it is 
enough to prove that \eqref{eq:kermap-ineq} implies condition (I).

Thus, assume \eqref{eq:kermap-ineq}. Given $\ell \in \{0,\dots,r\}$, let 
$i$ be critical for $D^1_{\bullet}$ with $\deg D^1_i=a^1_{\ell}$,
and choose $\ell_1,\ell_2,\ell_3,\ell_4$ as above.
Observe that $r_i \geq \#\{\ell':a^1_{\ell'}=i\} = \ell_2+1-\ell_1$
so that (I) implies that 
$r+1-\ell_1-\ell_3 \geq \ell_2+1-\ell_1$. It thus follows that
$r \geq \ell_2+\ell_3 \geq \ell+\ell_3$, so $r-\ell \geq \ell_3$.
Thus, we find
$$a^2_{r-\ell} \geq a^2_{\ell_3} \geq \deg D^2_{b-i},$$
giving (I), and completing the proof of the lemma.
\end{proof}

\section{The pseudocompact-type case}\label{sec:pseudocompact-type}

We conclude by generalizing the results of the previous section to
arbitrary curves of pseudocompact type, thereby providing a 
simultaneous generalization of the two-component case and the
compact-type case. As before, we start with combinatorial preliminaries.

\begin{notn}\label{not:multitree} If $\Gamma$ is a graph, let 
$\bar{\Gamma}$ be the graph obtained from $\Gamma$ by collapsing all
multiple edges to single edges, while leaving the vertex set unchanged.
We say $\Gamma$ is a \textbf{multitree} if $\bar{\Gamma}$ is a tree.
\end{notn}

Just as before we defined twists motivated by twisting at a component,
in the multitree case we define twists motivated by twisting on one side
or the other of the node(s) at which two components meet.

\begin{defn}\label{def:twist-node} If $\Gamma$ is a multitree,
and $(e,v)$ a pair of an edge $e$ and an adjacent 
vertex $v$ of $\bar{\Gamma}$, 
given an admissible multidegree $w$, we define the \textbf{twist} of
$w$ at $(e,v)$ to be obtained from $w$ as follows:
for each $\tilde{e}$ of $\Gamma$ over $e$, increase $\mu(\tilde{e})$ by 
$\sigma(\tilde{e},v)$. Now, decrease $w_{\Gamma}(v)$ by the number of 
$\tilde{e}$ for which $\mu(\tilde{e})$ had been equal to $0$, and for each 
$\tilde{e}$, if the new $\mu(\tilde{e})$ is zero, increase 
$w_{\Gamma}(v')$ by $1$, where $v'$ is the other vertex adjacent to $v$.
\end{defn}

Notice that if $v'$ is the other vertex adjacent to an edge $e$, then twisting 
at $(e,v')$ is inverse to twisting at $(e,v)$. In addition, we observe that 
the twist of $w$ at $(e,v)$ may be obtained as a sequence of
twists of $w$ at vertices $v'$, where $v'$ varies over the set of vertices
in the same connected component as $v$ in $\bar{\Gamma}\smallsetminus \{e\}$.
Conversely, twisting of $w$ at any $v$ can also be obtained as a composition
of twists at $(e,v)$, where $e$ varies over edges adjacent to $v$.

Throughout this section, all twists will be with respect to pairs $(e,v)$,
rather than vertices.

\begin{warn} Even though on a combinatorial level, twisting $w$ by $v$
can be obtained by a sequence of twists at different $(e,v)$, the same
does not hold on the level of the maps between the associated line bundles. 
\end{warn}

\begin{sit}\label{sit:concen-restrict} Suppose we are given a
multitree $\Gamma$, and an admissible multidegree $w_0$, and let 
$(w_v)_{v \in V(\Gamma)}$ be a
collection of elements of $V(G(w_0))$ such that:
\begin{Ilist}
\itm each $w_v$ is concentrated at $v$;
\itm for each $v,v' \in V(\bar{\Gamma})$ connected by an edge $e$, the
multidegree $w_{v'}$ is obtained from $w_v$ by twisting $b_{v,v'}$ times at 
$(e,v)$, for some $b_{v,v'}\in \ZZ_{\geq 0}$.
\end{Ilist}
\end{sit}

\begin{defn}\label{def:tree-graph}
In Situation \ref{sit:concen-restrict}, let 
$V(\bar{G}(w_0)) \subseteq V(G(w_0))$
consist of admissible multidegrees $w$ such that there exist 
$v,v' \in V(\bar{\Gamma})$ connected by some edge $e$, with $w$ obtainable 
from $w_v$ by twisting $b$ times at $(e,v)$, for some $b$ with
$0 \leq b \leq b_{v,v'}$.

There is an edge $\epsilon$ from from $w$ to $w'$ in $\bar{G}(w_0)$ if there
exist $(e,v)$ in $\bar{\Gamma}$ such that $w'$ is obtained from $w$ by
twisting at $(e,v)$.
\end{defn}

Thus, $\bar{G}(w_0)$ is a tree, obtained by subdividing every edge of 
$\bar{\Gamma}$ into $b_{v,v'}$ edges, and replacing each edge with a
pair of directed edges in opposite directions. Note that in general
the edges of $\bar{G}(w_0)$ need not be edges of $G(w_0)$, but can be 
thought of as ``compositions'' of edges of $G(w_0)$. However, in the case 
that $\Gamma$ has
only two vertices, we have that $G(w_0)$ and $\bar{G}(w_0)$ are both chains, 
with the only difference being that $\bar{G}(w_0)$ is bounded by $w_{v_1}$
and $w_{v_2}$, while $G(w_0)$ is unbounded. Thus, our notation is consistent
with that of Notation \ref{not:barG-2-comp}.

We now move on to the geometric definitions and statements.

\begin{defn}\label{def:pct}
Let $X_0$ be a projective nodal curve, with dual graph $\Gamma$.
$X_0$ is of \textbf{pseudocompact type} if $\Gamma$ is a multitree.
\end{defn}

\begin{sit}\label{sit:pct} In Situation \ref{sit:concen-restrict},
suppose also that our $\Gamma$ is obtained as the dual graph of a 
given projective nodal curve $X_0$.
\end{sit}

\begin{notn}\label{not:twist-divs} In Situation \ref{sit:pct},
for each pair $(e,v)$ of an
edge and adjacent vertex of $\bar{\Gamma}$, let 
$D^{(e,v)}_0,\dots,D^{(e,v)}_{b_{v,v'}+1}$ be the sequence of effective 
divisors on $Z_v$ defined by $D^{(e,v)}_0=0$, and for $i\geq 0$,
$$D^{(e,v)}_{i+1}-D^{(e,v)}_{i}
=\sum_{\scriptsize \begin{matrix}\tilde{e}\text{ over }e:\\ 
\sigma(\tilde{e},v)\mu_v(\tilde{e}) \equiv 
- i \pmod{\bn(\tilde{e})}\end{matrix}} 
P_{\tilde{e}},$$
where $P_{\tilde{e}}$ denotes the node of $X_0$ corresponding to $\tilde{e}$,
and $\mu_v(\bullet)$ is obtained from $w_v$.
\end{notn}

Our main result is the following.

\begin{thm}\label{thm:equiv} In the situation of Definition
\ref{def:lls}, suppose further that $X_0$ is of pseudocompact type,
and we are in Situation \ref{sit:pct}.
Then given a tuple $(\sL,(V^v)_{v \in V(\Gamma)})$,
for each pair $(e,v)$ in $\bar{\Gamma}$, let $a^{(e,v)}_0,\dots,a^{(e,v)}_r$
be the multivanishing sequence of $V^v$ along $D^{(e,v)}_{\bullet}$.
Then the following are equivalent:
\begin{alist}
\itm $(\sL,(V^v)_v)$ is a limit linear series;
\itm \eqref{eq:gluing-map-1} has kernel of dimension at least $r+1$ for
every $w \in V(\bar{G}(w_0))$;
\itm for any $e \in E(\Gamma)$, with adjacent vertices $v,v'$, we have:
\begin{Ilist}
\itm for $\ell=0,\dots,r$, if $a^{(e,v)}_{\ell}=\deg D^{(e,v)}_j$ with $j$ 
critical for $D^{(e,v)}_{\bullet}$, then
\begin{equation}\label{eq:eh-genl}
a^{(e,v')}_{r-\ell} \geq \deg D^{(e,v')}_{b_{v,v'}-j};
\end{equation}
\itm there exist bases $s^{(e,v)}_0,\dots,s^{(e,v)}_r$ of $V^v$ and
$s^{(e,v')}_0,\dots,s^{(e,v')}_r$ of $V^{v'}$ such that
$$\ord_{D^{(e,v)}_{\bullet}}s^{(e,v)}_{\ell} = a^{(e,v)}_{\ell},
 \quad \text{ for } \ell=0,\dots,r,$$
and similarly for $s^{(e,v')}_{\ell}$,
and for all $\ell$ with \eqref{eq:eh-genl} an equality, we have
$$\vp^{(e,v)}_{j}(s^{(e,v)}_{\ell})=s^{(e,v')}_{r-\ell}$$
when we consider $s^{(e,v)}_{\ell} \in V^v(-D^{(e,v)}_{j})$ and 
$s^{(e,v')}_{r-\ell} \in V^{v'}(-D^{(e,v')}_{b_{v,v'}-j})$, where 
$j$ is as in (I), and $\vp_j$ is
as in Proposition \ref{prop:twists-basic}.
\end{Ilist}
\end{alist}
\end{thm}

In (II) above, note that although Proposition \ref{prop:twists-basic} was
only stated for two-component curves, since we are only interested in
a given pair of adjacent vertices of $\Gamma$, the situation is no different
in our present more general case.

We first introduce some convenient notation. The following can be used to 
keep track of twisting at nodes:

\begin{notn}\label{not:tev} In Situation \ref{sit:concen-restrict},
given $w \in V(G(w_0))$, and $(e,v)$ adjacent in $\bar{\Gamma}$, 
let $t_{(e,v)}(w)$ be the number of twists at $(e,v)$ required to go
from $w_v$ to $w$ in a minimal number of twists.
\end{notn}

Note that $t_{(e,v)}(w)$ is well-defined, since the only way to cancel
a twist at $(e,v)$ is to twist at $(e,v')$, where $v'$ is the other vertex
adjacent to $v$. In addition, we have $t_{(e,v)}(w)+t_{(e,v')}(w)=b_{v,v'}$.

We can now define a notion of restriction of multidegrees to subcurves.
Of course, one can always restrict naively, but this turns out not to be
well behaved with respect to limit linear series, so instead we make the
following definition. 

\begin{defn}\label{def:deg-restrict} In Situation \ref{sit:concen-restrict},
let $X'_0$ be a connected subcurve of $X_0$, with dual graph $\Gamma'$. 
Then for any
$w \in V(G(w_0))$, define the \textbf{restriction} of $w$ to $X'_0$
as follows: starting from $w$, let $w'$ be the admissible multidegree
obtained by, for each pair $(e,v)$ in $\bar{\Gamma}$ where
$v \in \Gamma'$ but the other vertex $v'$ adjacent to $e$ is not in
$\Gamma'$, twisting $t_{(e,v)}(w)$ times at $(e,v')$. Then, the
restriction of $w$ to $X'_0$ is the naive restriction of $w'$.
\end{defn}

The reason for this choice of restriction, rather than the more naive
one, is that if we naively restrict an arbitrary $w$, it will no longer
be obtainable as a twist of the restrictions of the $w_v$. With our choice 
of restriction, even though we modify $w$, we will be able to understand 
the kernel of \eqref{eq:gluing-map-1} for a given $w$ in terms of the 
kernels of the restrictions to subcurves covering $X_0$; see the proof of 
Theorem \ref{thm:equiv} below.

Note that if $w \in V(\bar{G}(w_0))$, say between $w_v$ and $w_{v'}$, and
if $X'_0$ contains $Z_v$ and $Z_{v'}$, then in fact the restriction of
Definition \ref{def:deg-restrict} is simply the same as naive restriction.

\begin{proof}[Proof of Theorem \ref{thm:equiv}] First observe that because 
$V(\bar{G}(w_0)) \subseteq V(G(w_0))$, the implication (a) implies (b) is 
trivial. We will prove that (b) implies (c) and (c) implies (a), by 
induction on the number of components of $X_0$. The base case is that 
$X_0$ has two components, which is precisely Lemma 
\ref{lem:compare-2-comps}, together with Proposition 
\ref{prop:compare-restrict}.

Now, for the induction step, the basic observation is that condition (c) is
imposed on a pair of nodes at a time, so that (c) holds if and only if
for each pair $v_1,v_2$ of adjacent vertices of $\Gamma$, the restriction 
$(\sL_{w'_0}|_{Z_{v_1} \cup Z_{v_2}},(V^{v_1},V^{v_2}))$ also satisfies (c) 
for the curve $Z_{v_1} \cup Z_{v_2}$, where $w'_0$ is any element of 
$V(\bar{G}(w_0))$ lying between $w_{v_1}$ and $w_{v_2}$.
Note that $\deg \sL_{w'_0}|_{Z_{v_1} \cup Z_{v_2}}$ is not in general equal to
$d$, but is independent of the choice of $w'_0$.

Thus, to see that (b) implies (c), we suppose that
\eqref{eq:gluing-map-1} has kernel of dimension at least $r+1$ for
every $w \in V(\bar{G}(w_0))$, and we will show that if 
$v_1,v_2 \in V(\Gamma)$ are adjacent, then
$(\sL_{w'_0}|_{Z_{v_1} \cup Z_{v_2}},(V^{v_1},V^{v_2}))$ satisfies (c).
But suppose $w \in V(\bar{G}(w_0))$ lies between $w_{v_1}$ 
and $w_{v_2}$. Then note that $w$ agrees with both $w_{v_1}$ and $w_{v_2}$
away from $v_1$ and $v_2$ and the edges between them, so arguing as in 
Proposition \ref{prop:concentrated-inject},
the kernel of \eqref{eq:gluing-map-1} for $X_0$ 
injects into the kernel of \eqref{eq:gluing-map-1} for $Z_{v_1} \cup Z_{v_2}$
under restriction to $Z_{v_1} \cup Z_{v_2}$. 
Thus, by Lemma \ref{lem:compare-2-comps} we conclude that
$(\sL_{w'_0}|_{Z_{v_1} \cup Z_{v_2}},(V^{v_1},V^{v_2}))$ satisfies (c), as 
desired.

On the other hand, if (c) is satisfied, we prove the desired statement
by induction on the number of components. Given $w \in V(G(w_0))$, there
are two cases to consider.
First, if for some $(e,v)$, we
have $t_{(e,v)}(w)<0$, let $X'_0$ be the subcurve of 
$X_0$ corresponding to the connected component $\bar{\Gamma} \smallsetminus
\{e\}$ containing $v$.
Then if $w'$ is the multidegree obtained from $w$ in Definition
\ref{def:deg-restrict}, because $t_{(e,v)}(w)<0$, we have a map
$\sL_{w'} \to \sL_w$ which is injective on $\widetilde{X}'_0$; let
$Y$ be the subcurve of $\widetilde{X}_0$ on which it is injective,
and $Z$ the subcurve on which it vanishes.
Thus, $\widetilde{X}'_0 \subseteq Y$, and $\widetilde{X}_0=Y \cup Z$, 
and also $Y$ and $Z$ have no components in common. 
We thus have an inclusion $\sL_{w'}|_Y \to \sL_w|_Y$ whose image vanishes
at $Y \cap Z$, and it follows that we can extend by zero to get an inclusion
$$H^0(Y,\sL_{w'}|_{Y}) \hookrightarrow H^0(\widetilde{X}_0,\sL_w).$$
On the other hand, by construction we observe that $\sL_{w'}$ is trivial on
components on $Y$ not contained in $\widetilde{X}'_0$, 
so we have
$$H^0(\widetilde{X}'_0,\sL_{w'}|_{\widetilde{X}'_0})=
H^0(Y,\sL_{w'}|_{Y}),$$
inducing an inclusion
$$H^0(\widetilde{X}'_0,\sL_{w'}|_{\widetilde{X}'_0}) \hookrightarrow
H^0(\widetilde{X}_0,\sL_w).$$
Now, we have by hypothesis that (c)
is satisfied on $\widetilde{X}'_0$, so by the induction hypothesis,
the kernel of \eqref{eq:gluing-map-1} for $\widetilde{X}'_0$ has
dimension at least $r+1$ in multidegree $w'$, and using the above
inclusion, we get the same for $\widetilde{X}_0$ in multidegree $w$,
as desired.

The second case is that $t_{(e,v)}(w) \geq 0$ for all $(e,v)$, in which
case we necessarily have
$0 \leq t_{(e,v)}(w) \leq b_{v,v'}$. In this case, choose 
$v_1 \in V(\Gamma)$ which
is only adjacent to one other $v_2 \in V(\Gamma)$ (i.e., which is a leaf
of $\bar{\Gamma}$). Let $X'_0$ be the closure of the complement of 
$Z_{v_1}$ in $X_0$; then by hypothesis, (c) is satisfied for the restrictions
$(\sL_{w'_0}|_{Z_{v_1} \cup Z_{v_2}},(V^v,V^{v_2}))$ and
$(\sL_{w''_0}|_{X'_0},(V^v)_{v \neq v_1})$, where $w''_0$ is any
element of $V(\bar{G}(w_0))$ not lying between $w_{v_1}$ and $w_{v_2}$.
By the induction hypothesis, we conclude that \eqref{eq:gluing-map-1}
has kernel of dimension at least $r+1$ for $Z_{v_1} \cup Z_{v_2}$ in
multidegree $w'$ and for $X'_0$ in multidegree $w''$, where $w'$ and
$w''$ are the restrictions of $w$.
But because $0 \leq t_{(e,v)}(w) \leq b_{v,v'}$ for all $(e,v)$, 
the kernel of \eqref{eq:gluing-map-1} for $X_0$ in multidegree $w$
is simply the fibered product of the above two kernels over $V^{v_2}$,
and hence also has dimension at least $r+1$, as desired.
\end{proof}

\section{A smoothing theorem}\label{sec:smoothing}

In this section, we prove the following theorem, which says that -- just
as in the Eisenbud-Harris case -- when the space of limit linear series
on a curve of pseudocompact type has the expected dimension, then every 
limit linear series arises as the limit of linear series on smooth curves. 
In fact, our theorem is stronger even in the compact-type case, as it is 
not restricted to refined limit linear series. Our proof is fundamentally 
different from that of Eisenbud and Harris, although it still relies in 
the end on obtaining a lower bound on the dimension of a relative moduli 
space. The key ingredient is the theory of linked determinantal loci, 
developed in Appendix \ref{sec:link-det}. We also use a portion of
Theorem \ref{thm:equiv}, in essence to reduce to the two-component case.

\begin{thm}\label{thm:smoothing} 
Let $\pi:X \to B$ be a smoothing family, with special fiber $X_0$ a
curve of pseudocompact type.
Let $\bn$ be a chain structure on $X_0$, and
$\widetilde{\pi}:\widetilde{X} \to \widetilde{B}$ an extension of $\pi$ 
having fiber type $(X_0,\bn)$. Let $(\sO_v)_v$ be the induced enriched
structure on $X_0$. 

Given an admissible multidegree $w_0$ on the resulting $\widetilde{X}_0$, 
and $(w_v)_v$ as in Situation \ref{sit:concen-restrict},
if the moduli space $G^r_{\bar{w}_0}(X_0,\bn,(\sO_v)_v)$ has dimension
$\rho$ at a given point, then the corresponding limit linear series arises
as the limit of linear series on the geometric generic fiber of $\pi$.

More precisely, if $\pi:\widetilde{X}\to \widetilde{B}$ is any regular 
smoothing family of fiber type $(X_0,\bn)$, then the scheme
$\widetilde{G}^r_{\bar{w}_{0}}(\widetilde{X}/\widetilde{B},X_0,\bn,(\sO_v)_v)$
has universal relative dimension at least $\rho$ over $B$, and if the special
fiber $G^r_{\bar{w}_{0}}(X_0,\bn,(\sO_v)_v)$ has dimension exactly $\rho$
at a point, then 
$\widetilde{G}^r_{\bar{w}_{0}}(\widetilde{X}/\widetilde{B},X_0,\bn,(\sO_v)_v)$
is universally open at that point. If also the special fiber is geometrically
reduced at the given point, then 
$\widetilde{G}^r_{\bar{w}_{0}}(\widetilde{X}/\widetilde{B},X_0,\bn,(\sO_v)_v)$
is flat at that point.
\end{thm}

In the above, we use the relative dimension terminology introduced in
\cite{os21}.

\begin{proof}
The idea is to give a slightly different construction of the relative
limit linear series moduli space
$\widetilde{G}^r_{\bar{w}_{0}}(\widetilde{X}/\widetilde{B},X_0,\bn,(\sO_v)_v)$,
taking ideas from the proof of Theorem 5.3 of \cite{os8} and using the
linked determinantal loci developed in Appendix \ref{sec:link-det}. We can
work set-theoretically, since our goal is a dimension statement.
As in our earlier construction, start with the scheme 
$\Pic^{w_0}(\widetilde{X}_0)$, which is smooth over $B$ of relative 
dimension $g$, and let $\widetilde{\sM}$ be the universal line bundle, 
with $\widetilde{\sM}_w$ the induced line bundle in multidegree $w$ for
each $w \in V(G(w_0))$. Next, choose a sufficiently $\pi$-ample divisor $D$ 
on $\widetilde{X}$; using our sections of $\pi$, we may assume that
$D=\sum_{v \in V(\Gamma)} D_v$, where $D_v \cap X_0$ meets only $Z_v$. 
Note that we do not need to twist up on the exceptional components,
since they are rational and our multidegrees are always nonnegative on them.
We then have for each $w$ that 
$p_{1*} (\widetilde{\sM}_w(D))$ is locally free of rank
$d+\deg D + 1-g$, and commutes with base change. Let $G$ be the fibered 
product over $\Pic^{w_0}(\widetilde{X}_0)$ of the schemes
$G(r+1,p_{1*} (\widetilde{\sM}_{w_v}(D)))$, where $v$ ranges over
$V(\Gamma)$. This is thus smooth over $B$ of relative dimension 
\begin{multline*}g+|V(\Gamma)|(r+1)(d+\deg D+1-g-(r+1))\\
=g+|V(\Gamma)|(r+1)(d+\deg D-r-g).\end{multline*}

For each $v$, let $\sV_v$ be (the pullback to $G$ of) the universal subbundle 
of $p_{1*} (\widetilde{\sM}_{w_v}(D))$. Let $G'$ be the closed 
subset of $G$ obtained by imposing that for each $v$, the composed map 
$$\sV_v \to p_{1*} (\widetilde{\sM}_{w_v}(D)) \to
p_{1*} (\widetilde{\sM}_{w_v}(D)|_{D_v})$$
vanishes identically,
and by intersecting, for each $e \in E(\bar{\Gamma})$ having adjacent vertices 
$v,v'$, with the linked determinantal locus associated to the chain
$p_{1*} (\widetilde{\sM}_w(D))$ for $w$ between $w_v$ and $w_{v'}$
together with the subbundles $\sV_v$ and $\sV_{v'}$.
Then our key claim is that $G'$ is equal to
$\widetilde{G}^r_{\bar{w}_{0}}(\widetilde{X}/\widetilde{B},X_0,\bn,(\sO_v)_v)$.
Given the claim, we are done: the former conditions impose codimension at
most $(r+1)(\sum_{v} \deg D_v)=(r+1) \deg D$, 
and the latter impose, by Theorem \ref{thm:link-det}, codimension at most
\begin{multline*}|E(\bar{\Gamma})|(r+1)(d+\deg D+1-g-(r+1)) \\
=(|V(\Gamma)|-1)(r+1)(d+\deg D-r-g).\end{multline*}
Subtracting the above maximal codimensions from the relative dimension of
$G$, we are left with $g+(r+1)(d-r-g)=\rho$, and according to Corollary
5.1 of \cite{os21}, we find that
$\widetilde{G}^r_{\bar{w}_{0}}(\widetilde{X}/\widetilde{B},X_0,\bn,(\sO_v)_v)$
has universal relative dimension at least $\rho$ over $B$, as desired.
The assertions on universal openness and flatness in the case that 
the special fiber has dimension exactly $\rho$ at a point then follow from
Proposition 3.7 of \cite{os21}.

We are thus reduced to proving the claim. On the level of points, we 
analyze first the generic fiber $X_{\eta}$, and then the special fiber 
$\widetilde{X}_0$. Over the generic fiber, the maps between the $\sL_w$
are all isomorphisms, so the linked determinantal conditions in the
definition of $G'$ imply that the $V_v$ all map to one another under
these isomorphisms, and the condition that each $V_v$ vanish on $D_v$
implies that they all vanish on all of $D$. Thus, for a fixed choice of $v$,
we have that points of $G'$ on the generic fiber are all uniquely determined 
by a choice of $V_v$ contained in $\sL_{w_v}$, which is the same as 
$\widetilde{G}^r_{\bar{w}_{0}}(\widetilde{X}/\widetilde{B},X_0,\bn,(\sO_v)_v)$.
Next, on the special fiber, we are
asserting the following: given a line bundle $\sL$ of multidegree $w_0$ 
and a tuple $(V_v)_v$ with 
$V_v \subseteq \Gamma(\widetilde{X}_0,\sL_{w_v}(D))$,
if each $V_v$ vanishes on $D_v$, and for each $w \in V(\bar{G}(w_0))$
between $w_v$ and $w_{v'}$, the map
\begin{equation}\label{eq:gluing-two}
\Gamma(\widetilde{X}_0,\sL_w(D)) \to 
\Gamma(\widetilde{X}_0,\sL_{w_v}(D))/V_v \oplus 
\Gamma(\widetilde{X}_0,\sL_{w_{v'}}(D))/V_{v'} 
\end{equation}
has kernel of dimension at least $r+1$, then in fact each $V_v$ is 
contained in $\Gamma(\widetilde{X}_0,\sL_{w_v})$, and the map
\begin{equation}\label{eq:gluing-global}
\Gamma(\widetilde{X}_0,\sL_w) \to 
\bigoplus_v \Gamma(\widetilde{X}_0,\sL_{w_v})/V_v 
\end{equation}
has kernel of dimension at least $r+1$ for all $w \in V(G(w_0))$.
Our first observation is that for all $v,v'$, we must have $V_v$
mapping into $V_{v'} \subseteq \Gamma(\widetilde{X}_0,\sL_{w_{v'}}(D))$
under the natural twisting maps. Because the maps $\sL_{w_v} \to \sL_{w_{v'}}$
always factor as a sequence of such maps between adjacent vertices, it is
enough to prove this when $v,v'$ are adjacent. In this case, we consider
\eqref{eq:gluing-two} in the case $w=w_v$, noting that the kernel is
necessarily contained in $V_v$. Then our hypothesis implies that the 
kernel is all of $V_v$, and hence that $V_v$ maps into $V_{v'}$, as 
desired. 
Our next observation is that for $w \in V(\bar{G}(w_0))$, under
our hypotheses we have that the kernel of \eqref{eq:gluing-two} is
identified with the kernel of 
\begin{equation}\label{eq:gluing-global-prime}
\Gamma(\widetilde{X}_0,\sL_w(D)) \to 
\bigoplus_{v''} \Gamma(\widetilde{X}_0,\sL_{w_{v''}}(D))/V_{v''}.
\end{equation}
Indeed, this follows from the first observation, together with the fact 
that if $w$ lies between $w_v$ and $w_{v'}$, then for any $v''$ the map 
$\sL_w \to \sL_{w_{v''}}$ always factors through either $\sL_{w_v}$ or
$\sL_{w_{v'}}$.

It then follows that the kernel of \eqref{eq:gluing-global-prime} vanishes
on $D$ for each $w$, since
for each $v$, the map $\sL_{w} \to \sL_{w_{v}}$ is injective on 
$Z_{v}$, so if $V_{v}$ vanishes on $D_{v}$ the kernel of 
\eqref{eq:gluing-global-prime} vanishes on $D_v$ as well. Since the
$D_v$ are disjoint, we conclude that the kernel vanishes on $D$.
Considering the case $w=w_v$, we conclude in particular that each
$V_v$ vanishes on $D$, as desired. It follows that the kernel of
\eqref{eq:gluing-global-prime} is identified with the kernel of
\eqref{eq:gluing-global}, so we have proved the desired statement for
$w \in V(\bar{G}(w_0))$. Moreover, if we set $V^v$ to be the image of
$V_v$ in $\Gamma(Z_v,\sL^v)$, we see that the kernel of 
\eqref{eq:gluing-global} is identified with the kernel of 
\eqref{eq:gluing-map-1}, so the equivalence of (a) and (b) in Theorem
\ref{thm:equiv} then yields the desired statement for all $w \in V(G(w_0))$.
\end{proof}

\begin{rem}\label{rem:smooth-genl} Note that despite the pseudocompact type
hypothesis, our proof of the smoothing theorem was built around our general 
definition of limit linear series rather than the equivalent definition of
\S \ref{sec:pseudocompact-type}. In fact, we expect that a similar proof
should be possible in full generality, with the main difficulty being the
need for a much more general theory of linked Grassmannians. In our proof,
due to the special form of curves of pseudocompact type, we were able to
inductively reduce to what was, in essense, the ``two-component'' version
of the linked Grassmannian, but in general no such reduction is possible.
There is some evidence, in the form of examples and of parallel
results for local models of certain Shimura varieties (see, for instance,
Goertz \cite{go2}), that such a general theory of linked Grassmannians 
should exist, but we expect that it will be substantially more difficult 
than the special case we have used here.
\end{rem}

We conclude with a scheme structure comparison result involving the
construction carried out in the proof of Theorem \ref{thm:smoothing}. This
relates our construction to the related definitions for the higher-rank
case given in \S 4.2 of \cite{os20}, and more importantly will be used in
\cite{o-m1} to prove a comparison theorem in the rank-$1$, compact type
case between our scheme structure and the scheme structure given by
the Eisenbud-Harris definition. 

\begin{notn}\label{notn:grd-fam-2} Now suppose that we are in
the situation of Theorem \ref{thm:smoothing}, or of Theorem \ref{thm:equiv},
in which case we take $B=\widetilde{B}$ to be a point.
Let
$\widetilde{G}'^r_{\bar{w}_{0}}(\widetilde{X}/\widetilde{B},X_0,\bn,(\sO_v)_v)$ 
be the closed subscheme of the space
$\widetilde{P}^r_{w_{\bullet}}(\widetilde{X}/\widetilde{B},X_0,\bn,(\sO_v)_v)$
defined by the intersection of the $(r+1)$st 
vanishing loci of the maps \eqref{eq:lls-map-2}, as $w$ varies over 
$V(\bar{G}(w_0))$.
\end{notn}

Thus, \textit{a priori} we have that
$\widetilde{G}^r_{\bar{w}_{0}}(\widetilde{X}/\widetilde{B},X_0,\bn,(\sO_v)_v)$ 
is a closed subscheme of
$\widetilde{G}'^r_{\bar{w}_{0}}
(\widetilde{X}/\widetilde{B},X_0,\bn,(\sO_v)_v)$, 
and Theorem \ref{thm:equiv} tells us that they are supported on the same
subset.

\begin{prop}\label{prop:scheme-compare}
The moduli scheme 
$\widetilde{G}'^r_{\bar{w}_{0}}(\widetilde{X}/\widetilde{B},X_0,\bn,(\sO_v)_v)$ 
is proper over $\Pic^{w_0}(\widetilde{X}/\widetilde{B})$, and in the case
that $\pi$ is a smoothing family, its generic 
fiber is naturally identified with $G^r_d(X_{\eta})$. 

Moreover, the set-theoretic construction of
$\widetilde{G}^r_{\bar{w}_{0}}(\widetilde{X}/\widetilde{B},X_0,\bn,(\sO_v)_v)$ 
described in the proof of Theorem \ref{thm:smoothing} yields a scheme
structure agreeing with
$\widetilde{G}'^r_{\bar{w}_{0}}
(\widetilde{X}/\widetilde{B},X_0,\bn,(\sO_v)_v)$.
\end{prop}

\begin{proof} The proof of the first part is the same as for Proposition
\ref{prop:proper-generic}. For the second part, we need to show that the
set-theoretic analysis in the proof of Theorem \ref{thm:smoothing} works
on the level of $T$-valued points if we consider only $w \in V(\bar{G}(w_0))$.
Thus, suppose we are given a $T$-valued tuple $(\sL,(V_v)_v)$, where
each $V_v$ is a subbundle of $p_{1*} \sL_{w_v}(D)$, and for any
$w \in V(\bar{G}(w_0))$ between $w_v$ and $w_{v'}$, the map
\begin{equation}\label{eq:van-again}
p_{1*} \sL_w(D) \to ((p_{1*} \sL_{w_v}(D))/V_v) \oplus 
((p_{1*} \sL_{w_{v'}}(D))/V_{v'}) 
\end{equation}
has $(r+1)$st vanishing locus equal to $T$, and also that the composed maps 
$$V_v \to p_{1*} \sL_{w_v}(D) \to 
p_{1*} (\sL_{w_v}(D)|_{D_v}$$
are zero for each $v$. We want to show that in fact all the $V_v$ vanish
on all of $D$, and for all $w \in V(\bar{G}(w_0))$, the $(r+1)$st vanishing
locus of \eqref{eq:t-vald-2} is all of $T$.

First, given $v,v'$ adjacent, setting $w=w_v$ in \eqref{eq:van-again}, we
see by Proposition B.3.4 and Lemma B.2.3 (iv) of \cite{os20} that the
kernel must be equal to $V_v$,
and thus that $V_v$ maps into $V_{v'}$. Traversing 
$\bar{\Gamma}$ in this way we conclude that each $V_v$ maps into each
$V_{v'}$ for any $v' \neq v$. We then observe that for any $w$, and any
$v'$, the map 
$p_{1*} (\sL_w(D)|_{D_{v'}}) \to p_{1*} (\sL_{w_{v'}}(D)|_{D_{v'}})$
is an isomorphism, 
so since $V_v$ maps into $V_{v'}$ and $V_{v'}$ vanishes on $D_{v'}$, we
conclude that $V_v$ likewise vanishes on $D_{v'}$. Since $D=\sum_{v'} D_{v'}$,
we find that each $V_v$ vanishes on all of $D$, and may be considered as
a subbundle of $p_{1*} \sL_{w_v}$. Similarly, we see
that the kernel of \eqref{eq:van-again} is (universally) identified with
the kernel of 
\begin{equation}\label{eq:van-again-2}
p_{1*} \sL_w \to ((p_{1*} \sL_{w_v})/V_v) \oplus 
((p_{1*} \sL_{w_{v'}})/V_{v'}),
\end{equation}
so by Proposition B.3.2 of \cite{os20} we have that the $(r+1)$st vanishing 
loci of the two maps agree. But then, again using that each $V_v$ maps into
each other $V_{v'}$, and the map from $\sL_w$ to $\sL_{w_{v''}}$ factors
through $\sL_{w_v}$ or $\sL_{w_{v'}}$ if $w$ lies between $w_v$ and $w_{v'}$,
we see that the kernel of \eqref{eq:van-again-2} is also universally 
identified with the kernel of \eqref{eq:t-vald-2}, giving the desired
statement.

Note that neither the construction from Theorem \ref{thm:smoothing}
nor our analysis of its scheme
structure depended on $B$ being positive-dimensional, and in
particular we also conclude the desired statement in the case that $B$ is 
a point. 
\end{proof}

\appendix
\section{Linked determinantal loci}\label{sec:link-det}

In this appendix, we develop a theory of ``linked determinantal loci,''
which are in essence a determinantal locus analogue of the linked
Grassmannian developed in Appendix A of \cite{os8}. A preliminary 
definition is the following:

\begin{defn}\label{defn:s-linked} Let $S$ be a scheme, and $d,n$ be positive 
integers. Suppose that $\sE_1,\dots,\sE_n$ are vector bundles of rank $d$ 
on $S$ and we have morphisms
$$f_i:\sE_i \to \sE_{i+1},\quad f^i:\sE_{i+1} \to \sE_i$$
for each $i=1,\dots,n-1$. Given $s \in \Gamma(S,\sO_S)$, we say that
$\sE_{\bullet}=(\sE_i,f_i,f^i)_i$ is an {\bf $s$-linked chain} if the
following conditions are satisfied:
\begin{Ilist}
\itm For each $i=1,\dots,n$,
$$f_i \circ f^i = s \cdot \id, \text{ and } f^i \circ f_i = s \cdot \id.$$
\itm On the fibers of the $\sE_i$ at any point with $s=0$, we have that for
each $i=1,\dots,n-1$,
$$\ker f^i = \im f_i, \text{ and } \ker f_i = \im f^i.$$
\itm On the fibers of the $\sE_i$ at any point with $s=0$, we have that for
each $i=1,\dots,n-2$,
$$\im f_i \cap \ker f_{i+1}=(0),\text{ and }\im f^{i+1} \cap \ker f^i = (0).$$
\end{Ilist}
\end{defn}

This is precisely the condition required for the ambient chain of vector
bundles in the definition of a linked Grassmannian in \cite{os8}, although
the terminology was introduced later, in \cite{o-t1}. We then define:

\begin{defn}\label{defn:link-det} Let $\sE_{\bullet}$ be an $s$-linked
chain on a scheme $S$. Given $r>0$, suppose $\sF_1,\sF_n$ are rank-$r$
subbundles of $\sE_1$ and $\sE_n$ respectively. Then the associated
\textbf{linked determinantal locus} is the closed subscheme of $S$ on 
which the morphisms
\begin{equation}\label{eq:link-det}
\sE_i \to \left(\sE_1/\sF_1\right) \oplus \left(\sE_n/\sF_n\right)
\end{equation}
have rank less than or equal to $d-r$ for all $i=1,\dots,n$.
\end{defn}

In Definition \ref{defn:link-det}, the necessary morphisms $\sE_i \to \sE_j$
are obtained simply by composing the $f_i$ or $f^i$, as appropriate.

Thus, a linked determinantal locus is by definition an intersection of
$n$ determinal loci in $S$, for morphisms from vector bundles of rank $d$
to vector bundles of rank $2d-2r$. The standard codimension bound for
determinantal loci then implies that (each irreducible component of) a 
linked determinantal locus has codimension at most 
$n(d-(d-r))(2d-2r-(d-r))=nr(d-r)$. However, the
structure imposed by our hypotheses implies that in fact, the codimension
is far smaller. Our main theorem is the following.

\begin{thm}\label{thm:link-det} Each irreducible component of a linked 
determinantal locus has codimension at most $r(d-r)$ in $S$.
\end{thm}

\begin{rem}\label{rem:codim} Notice that set-theoretically, the
linked determinantal locus is the set of points of $S$ at which the 
kernel of \eqref{eq:link-det} has dimension at least $r$, or equivalently,
the set of points such that the fiber of $\sE_i$ contains at least an
$r$-dimensional space which maps into $\sF_1$ inside $\sE_1$ and into
$\sF_n$ inside $\sE_n$. In particular, the case $i=1$ implies that on
the linked determinantal locus, we must have $\sF_1$ mapping into $\sF_n$,
and the $i=n$ case implies that $\sF_n$ must map into $\sF_1$.

Now, in order to see that Theorem \ref{thm:link-det} is plausible,
consider points of $S$ over which $s$ is nonzero. On this locus,
all the maps are
isomorphisms, and our hypotheses imply that $\sF_1$ maps into $\sF_n$ if
and only if $\sF_n$ maps into $\sF_1$, and that moreover the linked
determinantal locus consists precisely of the points on which $\sF_1$
maps into $\sF_n$. Hence, on this locus it is clear that the 
codimension is at most $r(d-r)$, and we see that the interesting part of 
the theorem is the locus on which $s$ vanishes, or, crucially for our 
application to smoothing theorems, the global situation in which $s$ 
vanishes at some points but not others.
\end{rem}

The strategy of our proof parallels the proof of the corresponding statement
for determinantal varieties: we first consider the universal case and 
conclude the desired statement by realizing the linked determinantal locus
as the image of a linked Grassmannian, and then conclude the statement of
the theorem by pulling back from the universal case.

We next recall the definition of the linked Grassmannian.

\begin{defn}\label{def:lg} Let $S$ be a scheme, $\sE_{\bullet}$ an
$s$-linked chain on $S$, and $r>0$. Then the \textbf{linked Grassmannian}
$\LG(r,\sE_{\bullet})$ is the closed subscheme of
$$G(r,\sE_1) \times_S \dots \times_S G(r,\sE_n)$$
consisting of tuples $(\sF_1,\dots,\sF_n)$ such that for $i=1,\dots,n-1$
we have $f_i(\sF_i) \subseteq \sF_{i+1}$ and $f^i(\sF_{i+1})\subseteq \sF_i$.
\end{defn}

The relationship between linked Grassmannians and linked determinantal
loci is described by the following proposition.

\begin{prop}\label{prop:lg-project} Let $S_0$ be any scheme, and
$\bar{\sE}_{\bullet}$ an $s$-linked chain on $S_0$. Let 
$S=G(r,\bar{\sE}_1)\times_{S_0} G(r,\bar{\sE}_n)$, and let 
$\sE_{\bullet}$ be the pullback of $\bar{\sE}_{\bullet}$ to $S$, with
$\sF_1 \subseteq \sE_1$ and $\sF_n \subseteq \sE_n$ the pullbacks of
the universal bundles on $G(r,\bar{\sE}_1)$ and $G(r,\bar{\sE}_n)$
respectively. 

Then the linked determinantal locus associated to $\sE_{\bullet}$ and
$\sF_1,\sF_n$ is precisely the image of the linked Grassmannian 
$\LG(r,\bar{\sE}_{\bullet})$ under the projection morphism
$$G(r,\sE_1) \times_{S_0} \dots \times_{S_0} G(r,\sE_n) 
\to G(r,\sE_1) \times_{S_0} G(r,\sE_n).$$
\end{prop}

\begin{proof} It is clear from the definitions that the image of 
$\LG(r,\bar{\sE}_{\bullet})$ is contained in the linked determinantal
locus, so we need only prove the converse. Since the statement is
set-theoretic, we may work on the level of $k$-valued points with $k$
a field, and we see that what we want to prove is the following:
given $d$-dimensional $k$-vector spaces $E_1,\dots,E_n$, maps
$f^i$ and $f_i$ making an $s$-linked chain on $\Spec k$, and $r$-dimensional
subspaces $F_1 \subseteq E_1$ and $F_n \subseteq E_n$ such that the
kernel of \eqref{eq:link-det} has dimension at least $r$ for $i=1,\dots,n$,
then there exist choices of $r$-dimensional subspaces $F_i \subseteq E_i$
for $i=2,\dots,n-1$ which are linked by the $f_i$ and $f^i$.

Now, let $K_i \subseteq E_i$ be the kernel of \eqref{eq:link-det} for
$i=2,\dots,n-1$. Then by hypothesis, $\dim K_i \geq r$ for all $i$, and
it is also clear that $f_i (K_i) \subseteq K_{i+1}$ and 
$f^i (K_{i+1}) \subseteq K_i$ for all $i$.
We claim that as long as $\dim K_i > r$ for some $i$, we can replace some
$K_i$ by a proper subspace while preserving the above conditions;
iterating this process yields the desired statement. Now, let $i$ be
minimal such that $\dim K_i > r$; we claim that the span of the images
of $K_{i-1}$ and $K_{i+1}$ in $K_i$ must be strictly smaller than $K_i$.
Indeed, by condition (III) of $s$-linkage, the image of $K_{i+1}$ in $K_i$ 
also injects into $K_{i-1}$, but maps into the kernel of $f_{i-1}$.
Because $\dim K_{i-1}=r$, we conclude that the span of the images of 
$K_{i-1}$ and $K_{i+1}$ in $K_i$ must have dimension at most $r$, so 
we can replace $K_i$ by any $r$-dimensional subspace containing this
span; this will preserve the linkage condition, and thus proves the claim.
\end{proof}

We next need to set up the relevant universal spaces. We have the following:

\begin{prop}\label{prop:matrix-prod-smooth} Given $d>0$, let $\bar{U}_d$ be 
the scheme of pairs of $d \times d$ matrixes $A$ and $B$ over $\ZZ[t]$ such 
$$AB=BA=t I_d.$$
Let $U_d$ be the open subscheme of $\bar{U}_d$ on which 
$$\rk A + \rk B \geq d.$$

Then $U_d$ is smooth over $\Spec \ZZ[t]$ of relative dimension $d^2$.
\end{prop}

\begin{proof} We first observe that the fibers are smooth of dimension 
$d^2$: over points with $t \neq 0$, this is clear, as $U_d$ is simply 
isomorphic to $\GL_d$; on the other hand, where $t=0$ Strickland
\cite{st4} shows that $U_d$ is reduced of dimension $d^2$,
and if we fix the ranks of $A$ and $B$ (necessarily adding to $d$), we 
obtain an open subset of $U_d$
which is an orbit of the action of $\GL_d \times \GL_d$, and must 
therefore be smooth. 

Thus, it is enough to show that $U_d$ is flat over $\Spec \ZZ[t]$.
For this, we appeal to Lemma 4.3 of \cite{o-h1}, which asserts that it
is enough to check that for any base change of $U_d$ to $\Spec R$ with
$R$ a discrete valuation ring, no component of the base change is 
supported in the special fiber. This then amounts to the assertion
that if we are given a discrete valuation ring $R$, and an element $x$ of
$R$, that the scheme of pairs of $d \times d$ matrices $A,B$ over $R$ with
$AB=BA=x I_d$ and with $\rk A + \rk B \geq d$ does not have components 
supported over the closed point of $R$. But if we are given such $A,B$ 
over the residue field $k$ of $R$, with $\rk A=d_1$ and $\rk B=d_2$, there 
are two cases to consider: if $x$ is a unit, then $A$ and $B$ are 
invertible, so we may choose any lift of $A$ to $R$, and set $B=x A^{-1}$. 
On the other hand, if $x$ maps to $0$ in $k$, then up to change of basis on 
both sides, we may assume $A$ is diagonal with the first $d_1$ diagonal 
entries equal to $1$, and the remaining entries $0$, and $B$ is diagonal 
with the first $d-d_2=d_1$ entries equal to $0$, and the remaining entries 
equal to $1$. We may then lift to $R$ simply by replacing the diagonal 
$0$s with $x$.
This shows that every point in the closed fiber is in fact contained in a
section, yielding the desired statement.
\end{proof}

Finally, we recall the relevant theorem on linked Grassmannians from
\cite{os8}.

\begin{thm}\label{thm:lg}
Suppose that $S$ is integral and Cohen-Macaulay, and $\sE_{\bullet}$ is
an $s$-linked chain on $S$. Then every component of $\LG(r,\sE_{\bullet})$ 
has codimension $(n-1)r(d-r)$ inside 
$G(r,\sE_1) \times_S \dots \times_S G(r,\sE_n)$,
and if $s$ is nonzero, then $\LG(r,\sE_{\bullet})$ is irreducible.
\end{thm}

We are now ready to prove our main theorem.

\begin{proof}[Proof of Theorem \ref{thm:link-det}] Let $T$ be the
product of $n-1$ copies of $U_d$ over $\Spec \ZZ[t]$, and let $S_0^{\univ}$
be the open subscheme of $T$ on which $\ker A_{i+1} \cap \im A_i=(0)$
and $\ker B_{i} \cap \im B_{i+1}=(0)$ for $i=1,\dots,n-2$. Then we have
an $s$-linked chain $\sE_{\bullet}^{\univ}$ on $S_0^{\univ}$ (with $s=t$) 
by taking
$n$ copies of the trivial bundle, and using the $A_i$ and $B_i$ to define
our maps. Let $S^{\univ}$ be obtained from $S_0^{\univ}$ as in
Proposition \ref{prop:lg-project}. We claim that it is enough to prove the
theorem for the corresponding linked determinantal locus on $S^{\univ}$.
Indeed, given any $S$ and $\sE_{\bullet}$, the theorem is local on $S$,
so we may assume that the $\sE_i$ are trivialized, and our $s$-linked
chain and subbundles $\sF_1$ and $\sF_n$ then induce a morphism to
$S^{\univ}$ under which they are obtained as the pullbacks of
$\sE_{\bullet}^{\univ}$ and the universal subbundles. Moreover, under
this morphism we have that the linked determinantal locus on $S$ is the
preimage of the linked determinantal locus on $S^{\univ}$. Now, by
Proposition \ref{prop:matrix-prod-smooth} we have that $S_0^{\univ}$
and hence $S^{\univ}$ is smooth over $\Spec \ZZ[t]$, and hence regular,
and it then follows by Theorem 7.1 of \cite{ho1} that if every component
of the linked determinantal locus in $S^{\univ}$ has codimension at most
$r(d-r)$, then the same is true in $S$.

But according to Proposition \ref{prop:lg-project}, the linked determinantal
locus in $S^{\univ}$ is the image of the linked Grassmannian 
$\LG(r,\sE_{\bullet}^{\univ})$ over $S^{\univ}_0$. By Theorem \ref{thm:lg},
we know that $\LG(r,\sE_{\bullet}^{\univ})$ is irreducible of codimension
$(n-1)r(d-r)$, and it is clear that it maps generically finitely onto its 
image in $S^{\univ}$, since for $t \neq 0$ the subbundle $\sF_1$ uniquely 
determines all the other subbundles. We thus conclude by
Proposition 5.6.5 of \cite{ega42} that
the image -- that is, the linked determinantal locus -- has codimension 
$r(d-r)$, as desired.
\end{proof}

\bibliographystyle{amsalpha}
\bibliography{gen}

\end{document}